\documentclass[11pt,leqno]{amsart}
\usepackage{amssymb,amsmath,amsthm,amsfonts}
\usepackage {latexsym}
\usepackage{bbm}

\allowdisplaybreaks
\newcommand{\nc}{\newcommand}
\nc{\les}{\lesssim}
\nc{\nit}{\noindent}
\nc{\nn}{\nonumber}
\nc{\D}{\partial}
\nc{\diff}[2]{\frac{d #1}{d #2}}
\nc{\diffn}[3]{\frac{d^{#3} #1}{d {#2}^{#3}}}
\nc{\pdiff}[2]{\frac{\partial #1}{\partial #2}}
\nc{\pdiffn}[3]{\frac{\partial^{#3} #1}{\partial{#2}^{#3}}}
\nc{\abs}[1] {\lvert #1 \rvert}
\nc{\cAc}{{\cal A}_c}
\nc{\cE}{{\cal E}}
\nc{\cF}{{\cal F}}
\nc{\cP}{{\cal P}}
\nc{\cV}{{\cal V}}
\nc{\cQ}{{\cal Q}}
\nc{\cGin}{{\cal G}_{\rm in}}
\nc{\cGout}{{\cal G}_{\rm out}}
\nc{\cO}{{\cal O}}
\nc{\Lav}{{\cal L}_{\rm av}}
\nc{\cL}{{\cal L}}
\nc{\cB}{{\cal B}}
\nc{\cZ}{{\cal Z}}
\nc{\cR}{{\cal R}}
\nc{\cT}{{\cal T}}
\nc{\cY}{{\cal Y}}
\nc{\cX}{{\cal X}}
\nc{\cXT}{{{\cal X}(T)}}
\nc{\cBT}{{{\cal B}(T)}}
\nc{\vD}{{\vec \mathcal{D}}}
\nc{\efield}{\mathcal{E}}
\nc{\vE}{{\vec \efield}}
\nc{\vB}{{\vec \mathcal{B}}}
\nc{\vH}{{\vec \mathcal{H}}}
\nc{\ty}{{\tilde y}}
\nc{\tu}{{\tilde u}}
\nc{\tV}{{\tilde V}}
\nc{\Pc}{{\bf P_c}}
\nc{\bx}{{\bf x}}
\nc{\bX}{{\bf X}}
\nc{\bXYZ}{{\bf XYZ}}
\nc{\bY}{{\bf Y}}
\nc{\bF}{{\bf F}}
\nc{\bS}{{\bf S}}
\nc{\dV}{{\delta V}}
\nc{\dE}{{\delta E}}
\nc{\TT}{{\Theta}}
\nc{\dPsi}{{\delta\Psi}}
\nc{\order}{{\cal O}}
\nc{\Rout}{R_{\rm out}}
\nc{\eplus}{e_+}
\nc{\eminus}{e_-}
\nc{\epm}{e_\pm}
\nc{\eps}{\varepsilon}
\nc{\vnabla}{{\vec\nabla}}
\nc{\G}{\Gamma}
\nc{\w}{\omega}
\nc{\mh}{h}
\nc{\mg}{g}
\nc{\vphi}{\varphi}
\nc{\tlambda}{\tilde\lambda}
\nc{\be}{\begin{equation}}
\nc{\ee}{\end{equation}}
\nc{\ba}{\begin{eqnarray}}
\nc{\ea}{\end{eqnarray}}

\nc{\g}{\gamma}
\nc{\ol}{\overline}

\newtheorem{theorem}{Theorem}[section]
\newtheorem{lemma}[theorem]{Lemma}
\newtheorem{prop}[theorem]{Proposition}
\newtheorem{corollary}[theorem]{Corollary}
\newtheorem{defin}[theorem]{Definition}

\nc{\pT}{\partial_T}
\nc{\pz}{\partial_z}
\nc{\pt}{\partial_t}
\nc{\la}{\langle}
\nc{\ra}{\rangle}
\nc{\infint}{\int_{-\infty}^{\infty}}
\nc{\halfwidth}{6.5cm}
\nc{\figwidth}{10cm}
\newcommand{\f}{\frac}

\nc{\nlayers}{L} \nc{\nsectors}{M}
\nc{\indicator}{\mathbf{1}}
\nc{\Rhole}{R_{\rm hole}}
\nc{\Rring}{R_{\rm ring}}
\nc{\neff}{n_{\rm eff}}
\nc{\Frem}{F_{\rm rem}}
\nc{\R}{\mathbb R}
\nc{\Z}{\mathbb Z}
\nc{\DD}{\Delta}
\nc{\cD}{\mathcal D}
\nc{\lnorm}{\left\|}
\nc{\rnorm}{\right\|}
\nc{\rnormp}{\right\|_{\ell^{p,\eps}}}
\nc{\rar}{\rightarrow}
\date{\today}
\title [A WEIGHTED ESTMATE WHEN ZERO IS THE RESONANCE OF THE FIRST KIND]{ A WEIGHTED ESTIMATE FOR TWO DIMENSIONAL SCHR\"ODINGER, MATRIX SCHR\"ODINGER AND WAVE EQUATIONS WITH RESONANCE OF FIRST KIND AT ZERO ENERGY}
\author{EBRU TOPRAK}
\thanks{The author was partially supported by NSF grants DMS-1201872 and DMS-1501041, and would like to thank Burak Erdo\u{g}an for the financial support and for his expert advise. }
\begin{document}
\address{Department of Mathematics \\
University of Illinois \\
Urbana, IL 61801, U.S.A.}
\email{toprak2@illinois.edu}
\maketitle
\thispagestyle{empty}

\begin{abstract}
 We study the two dimensional Schr\"odinger operator, $H=-\Delta+V$, in the weighted $L^1(\R^2) \rightarrow L^{\infty}(\R^2)$ setting  when there is a resonance of the first kind at zero energy. In particular, we show that if $|V(x)|\les \la x \ra ^{-3-}$ and there is only s-wave resonance at zero of $H$, then   
 $$ \big\| w^{-1} \big( e^{itH}P_{ac} f - {\f 1 t } F f \big) \big\| _{\infty} \leq \frac {C} {|t| (\log|t|)^2 } \|wf\|_1,\,\,\,\,\,\,|t|>2,$$
 with $w(x)=\log^2(2+|x|)$. Here $Ff=c \psi\la f,\psi \ra$, where $\psi$ is an s-wave resonance function. We also extend this result to matrix Schr\"odinger and wave equations with potentials under similar conditions. 
\end{abstract}

\section{ Introduction}
Recall the propagator of  the free Schr\"odinger equation:
\begin{align} \label{free}
 e^{-it\Delta} f(x) = \frac{1}{(4\pi i t )^{n/2}} \int_{\R^n} e^{{-i |x-y|^2}/{4t}} f(y) dy 
   \end{align} 
which satisfies the dispersive estimate 
     \begin{align}
  \label{decay}
  \|e^{-it\Delta} f \|_\infty \les  t^{-n/2} \|f\|_1 
  \end{align}
  for any $n\geq 1$.
There are many works concerning the validity of such an estimate for the perturbed Schr\"odinger operator $H=-\Delta + V $ where $V(x)$ is a real-valued and bounded potential with sufficient decay at infinity, see for example \cite{JSS,RodSch,GS,goldE,GV, Yaj,FY,CCV,EG1}. 
 
Since $H$ may have eigenvalues on $(-\infty,0]$, the inequality \eqref{decay} cannot hold in general. Therefore, we consider  $e^{itH}P_{ac}(H)$ where $P_{ac}(H)$ is the orthogonal projection onto the absolutely continuous subspace of $L^2(\R^n)$. It was observed that the time decay of the operator $e^{itH}P_{ac}(H)$ is affected by resonances or an eigenvalue at zero energy (see, e.g., \cite{Rauch,JenKat,Mur,Jen,Jen2,ES2,Yaj3, goldE, Bec, EG}). 

Recall that, in two dimensions, a distributional solution of  $H\psi = 0$   is called an s-wave resonance if $\psi \in L^{\infty}(\R^2)$ but $\psi \notin L^p(\R^2)$ for any $p<{\infty} $,  and it is called a p-wave resonance if $\psi \in L^p(\R^2)$ for $ 2<p\leq \infty $, but $\psi\not\in L^2(\R^2)$. We also say there is a resonance of the first kind at zero if there is only an  s-wave  resonance at zero but there are no p-wave resonances or an eigenvalue at zero. It is important to recall that in this case there is only one s-wave resonance function upto a multiplicative constant. There are similar definitions for resonance in dimensions $n=1, 3$, $4$, and there are no zero energy resonances in dimensions $n\geq 5$ . 
     
  We note that by these definitions  constant function $1$ is an s-wave resonance  in dimension two for the free Schr\"odinger operator. In addition, using the formula (\ref{free}), one  can easily prove that 
\begin{align}\label{freemain}
\big\|  w^{-1} \big( e^{-it\Delta} f -\frac1{t} \psi \la f,\psi\ra  \big)\big\|_{L ^{\infty}(\R^2)} \les \frac1{t\log^2 t}  \|wf\|_{L^1(\R^2)}
\end{align} 
where $w(x)= \log^2(2+|x|)$, and $\psi$ is the resonance function $\psi(x)=1$. This suggests that the perturbed Schr\"odinger evolution should satisfy a similar weighted estimate with an integrable decay rate in the case of an s-wave resonance. Indeed, our main result in this paper is
\begin{theorem} \label{main1} Let $|V (x)| \les \la x \ra ^{-2\beta} $ for some $ \beta > {3/2}.$ If there is a resonance of the first kind at zero for $H = -\Delta + V$, then we have
 $$ \big\| w^{-1} \big( e^{itH}P_{ac} f - {\f 1 t } F f  \big) \big\|_{L ^{\infty}(\R^2)} \leq \frac {C} {|t| (\log|t|)^2 } \|wf\|_{L_1(\R^2)}  $$
for $|t|>2$ where $F$ is a rank $1$ operator  and  $w(x)=\log^2(2+|x|)$.
Furthermore, $Ff=c \psi\la f,\psi \ra$, where $\psi$ is an s-wave resonance function.
\end{theorem}
  
We then extend this result to matrix Schr\"odinger operator and to the low-energy evolution of the solution of the two-dimensional wave equation, [see Theorem~\ref{main2} and Theorem~\ref{main3}].

The dispersive estimate
\begin{align}
 \label{n/2}
     \| e^{itH}P_{ac}f\|_{L^\infty(\R^n)} \leq C |t|^{-\f n 2 } \|f\|_{L_1(\R^n)}
     \end{align}
 in dimensions one and two were studied in \cite{GS,Sc2,goldbeg1,KZ,Mou,EG,EGw}. In fact, (\ref{n/2}) is established by Goldberg-Schlag for $n=1$ in \cite{GS} and Schlag for $n=2$ in \cite{Sc2} assuming zero is regular,  that is when there is neither a resonance nor an eigenvalue at zero.  The result in dimension two is then improved by Erdo\u{g}an-Green to a more general case. They showed the same estimate when there is a resonance of the first kind at zero.

 The main concern for these estimates is that they are not integrable in time  at infinity. An estimate which is integrable at infinity is very useful in the study of nonlinear asymptotic stability of (multi) solitons in lower dimensions. See \cite{scns,KZ, Mur, BP, SW,PW,Wed} for other applications of weighted dispersive estimates to nonlinear PDEs.

  One of the earliest integrable decay rates is established by Murata in weighted $L^2$ spaces. In \cite{Mur}, [Theorem 7.6], Murata proved the following statement in polynomially weighted spaces by assuming sufficient decay on $V$. If zero is a regular point of the spectrum, then for $|t| > 2$
          \begin{align}
          \label{mur1}
           \| w_1^{-1} e^{itH} P_{ac}(H)f\|_{L^2(\R)} \leq  C t^{-\f 32} \|w_1f\|_{L^2(\R)}, 
              \end{align}
           \begin{align} \label{mur2}
                \| w_2^{-1} e^{itH} P_{ac}(H)f\|_{L^2(\R^2)} \leq \frac {C} {|t| (\log|t|)^2 } \|w_2f\|_{L^2(\R^2)}.
                      \end{align}
                      
   In \cite{Scs} Schlag improved Murata's 1-d result (\ref{mur1}) to weighted $L^1 \rightarrow L^{\infty}$ setting and he showed if zero is regular then 
        $$ \| w^{-1} e^{itH} P_{ac}(H)f\|_{L^{\infty}(\R)} \leq  C t^{-\f 32} \|w^{-1}f\|_{L^1(\R^2)} $$
for $w(x)=\la x  \ra $ provided that $|V| \les \la x \ra ^{-4}$. 
  
  Constant functions being resonance in dimension one together with (\ref{freemain}) led Goldberg to ask whether a similar estimate as in Theorem~\ref{main1} can be obtained when zero is not regular. Specifically, Goldberg showed that if $ (1+|x|)^4 V \in L^1(\R) $ then
      \begin{align} \label{mikegold}
      \| (1+|x|)^{-2} (e^{itH} P_{ac}{H} -(-4 \pi i t)^{-\f 12 })F)f \|_{L^{\infty}} \les t^{-\f 32} \|(1+|x|^2) f\|_{1} 
      \end{align}
where $F$ is a projection on a bounded function $f_0$ satisfying $ Hf_0 = 0 $ and $\lim_{x\to \infty} (|f_0(x)|+|f_0(-x)|)=2$.
  
  Murata's (\ref{mur2}) result for dimension two was also improved by Erdo\u{g}an-Green. In \cite{EGw}, it was proved that if zero is regular then 
         \begin{align} \label{ergr}
          \| w^{-1} e^{itH} P_{ac}(H)f\|_{L^1(\R^2)} \leq \frac {C} {|t| (\log|t|)^2 } \|w f\|_{L^{\infty}(\R^2)}
          \end{align}
  for $w(x)= \log(2+x^2)$ and $|V| \les \la x \ra ^{-\beta} $ for $\beta>3 $.           
 Theorem~\ref{main1} above was  motivated by Goldberg's result \eqref{mikegold} and  Erdo\u{g}an-Green's result \eqref{ergr}.     
     
     There have been also studies of the Schr\"odinger operator in dimensions $n=3,4$ and $n>4$. For more details about these dimensions one can see \cite{Scs,ES2,ES3,EGG,GG1,GG2}
  
   We define the resolvent operator as $R_V^{\pm}(\lambda^2) = \lim_{\epsilon \to 0}(H - ( \lambda^2 \pm i\epsilon))^{-1}$. By Agmon's limiting absorption principle \cite{agmon}, this limit is  well-defined as an operator from $L^{2,\sigma}$ to $ \mathcal{H}_{2,-\sigma}$ for $\sigma > {\f 1 2}$ where $L^{2,\sigma} =  \{ f : \la x \ra ^{\sigma}f \in L^2(\R^n)\}$ and $\mathcal{H}_{2,\sigma}=\{ f : D^{\alpha}f\in L^{2,\sigma}(\R^n), \hspace{5mm}  0\leq |\alpha| \leq 2 \} $. The proof of Theorem~\ref{main1} 
   relies on expansions of $R_V^\pm$ around zero energy and Stone's formula for self-adjoint operators:  
\begin{align}
\label{stone}
 \ e^{itH} P_{ac}(H) \chi(H) f(x) =\frac1{2\pi i}  \int_0^{\infty} e^{it\lambda^2} \lambda  \chi( \lambda )  [R_V^+(\lambda^2)-R_V^{-}(\lambda^2)]  f(x) d\lambda,\,\,\,f \in \mathcal S(\R^2).
 \end{align}
   Here $\chi$ is an even smooth cut-off function supported in $[-\lambda_1,\lambda_1]$ for a fixed sufficiently small $\lambda_1>0$ and it is equal to one if  $|\lambda| \leq \frac{\lambda_1}{2}$. We note that, in our analysis $V$ has enough decay to ensure that $H$ has finitely many eigenvalues of finite multiplicity on $(-\infty, 0]$, with $\sigma_{ac}(H)=[0,\infty)$, see \cite{RS1}.

   We then extend our result to the non-self adjoint matrix Schr\"odinger operator. The non self-adjoint matrix Schr\"odinger operator  is defined as
 \begin{align} \label{nonself}
                  \mathcal{H}=\mathcal{H}_0+ V =  \left[\begin{array}{cc} -\Delta + \mu & 0 \\ 0 & \Delta - \mu \end{array}\right]+ \left[\begin{array}{cc} -V_1 & -V_2\\ V_2 & V_1 \end{array}\right]
     \end{align}
on $L^2(\R^2)\times L^2(\R^2)$ where $ \mu > 0$ and $V_1$, $ V_2$ are real valued potentials.
Note that if we diagonalize $\mathcal{H}$ with the matrix $ \left[\begin{array}{cc} 1 & i \\ 1 & -i  \end{array}\right] $ we obtain $\left[\begin{array}{cc} 0 & iL\_ \\ iL_+ & 0 \end{array}\right] $. That matrix together with Weyl's criterion gives us  $\sigma_{ess}(\mathcal{H})= (-\infty , - \mu] \cup [\mu , \infty) $  assuming some decay on $V_1$ and $ V_2$. 

We need the following assumptions, \\
A1) - $\sigma_3 V $ is a positive matrix  where  $\sigma_3$ is the Pauli spin matrix $$  \sigma_3 = \left[\begin{array}{cc} 1 & 0 \\ 0 & -1 \end{array}\right] , $$
A2) $L \_ = - \Delta + \mu - V_1 + V_2 \ \geq0 $, \\
A3) $|V_1| + |V_2| \les \la x \ra ^{-\beta} $ for some $\beta > 3 $, \\
A4) There are no embedded eigenvalues in $(-\infty , - \mu) \cup (\mu , \infty) $.

It is known that the first three assumptions are to hold in the case when the Schr\"odinger equation is linearized about a positive ground state standing wave $\psi(t,x)=e^{it\mu}\phi(x)$. We need the fourth assumption to be able to define the spectral measure from $X_\sigma$ to $X_{-\sigma}$ where $X_\sigma = L^{2,\sigma} \times L^{2,\sigma}$. For more details one can see \cite{ES3} and \cite{EGm}.

Dispersive estimates for the operator (\ref{nonself}) is studied in \cite{Cuc,RodSch1,scs, ES2, Cuct, mar, green}. In the case when thresholds are  regular the following result is obtained in dimension two.
\begin{theorem}[Theorem 1.1 in \cite{EGm}] Under the assumptions $A1)-A4)$ if $\pm \mu$ are regular points of   $\mathcal{H}=\mathcal{H}_0+ V$ we have
 $$ \|e^{itH} P_{ac} f \|_{L^{\infty} \times L^{\infty} } \les {\f 1 {|t|}} \|f\|_{L^1 \times L^1} $$
 and
\begin{align*}
         \| w^{-1}  e^{itH}P_{ac} f \| _{L^{\infty} \times L^{\infty}} \les \frac {1} {|t|(\log|t|)^2} \|wf\|_{L^1 \times L^1} ,   \hspace{3mm}  |t|>2 
                  \end{align*}
where  $w(x)=\log^2(2+|x|)$.
  \end{theorem}
Our main result for the matrix Schr\"odinger operator is Theorem~\ref{main2}.
\begin{theorem}\label{main2} Under the conditions A1)-A4), if there is a resonance of the first kind at the threshold $ \mu$ then we have
 $$ \|e^{itH} P_{ac} f \|_{L^{\infty} \times L^{\infty} } \les {\f 1 {|t|}} \|f\|_{L^1\times L^1} $$
 and
 \begin{align*}
         \| w^{-1} \big( e^{itH}P_{ac} f - {\f 1 t } \mathfrak{F} f \big) \| _{L^{\infty} \times L^{\infty}} \les \frac {C} {|t|(\log|t|)^2} \|wf\|_{L^1 \times L^1} ,   \hspace{3mm}  |t|>2. 
                  \end{align*}
 Here $\mathfrak{F}(x,y)=c \psi(x)\sigma_3\psi(y)$, where $\psi(x)$ is an s-wave resonance  and $w(x)=\log^2(2+|x|)$.
 
 A similar statement holds if there is a resonance of first kind at $-\mu$.  
  \end{theorem}

  The resolvent expansions we obtain to prove Theorem~\ref{main1}  for Schr\"odinger evolution are also applicable to   the two-dimensional wave equation with a potential. Recall that the perturbed wave equation is given as 
\begin{equation}
 \label{wave}
   u_{tt} - \Delta u + V(x)u=0, \hspace{5mm} u(x,0)=f(x), \hspace{5mm} u_t(x,0)=g(x)
    \end{equation}
with the solution formula
       $$ u(x,t)=\cos(t\sqrt{H}) f(x) + \frac{\sin(t\sqrt{H})}{\sqrt{H}} g(x) $$
for $f \in W^{2,1}$ and $ g\in W^{1,1}$. By  Stone's formula, we have the representations 
\begin{align}
\label{stonecos}
 &\cos(t\sqrt{H}) P_{ac}f(x)=\frac{1}{2\pi i} \int_0^{\infty} \cos(t \lambda)\lambda [ R_V^{+}(\lambda^2)-R_V^{-}(\lambda^2) ] f(x) d\lambda \\
& \frac{\sin(t\sqrt{H})}{\sqrt{H}} P_{ac}g(x)=\frac{1}{2\pi i} \int_0^{\infty} \sin(t\lambda) [ R_V^{+}(\lambda^2)-R_V^{-}(\lambda^2) ] g(x) d\lambda. \label{stonesin}
     \end{align}

For the low energy, that is when $0< \lambda \ll 1 $ this representation leads us to a similar result as in Theorem~\ref{main1}. On the other hand, for the large energy, $\lambda \gtrsim 1 $,  one needs regularizing powers of $\la H \ra ^ {-\alpha} $ for some $ \alpha > 0$ which reflects the loss of derivatives of initial data, see., e.g. \cite{green}. 

In dimension two, these type of estimates are studied in  \cite{CCV1,Mou, greenw}. In \cite{greenw}, Green proved that  if there is a resonance of first kind at zero, then
 \begin{align} \label{lem:greenwave}
    \begin{split}
   &\big\|  \cos(t\sqrt{H})\la H \ra^{-3/4-}  P_{ac} f(x) \big\| _{L^{\infty}}\les |t|^{-\f1{2}} \|f\|_{L^1}, \\
   & \big\| \frac{\sin(t\sqrt{H})}{\sqrt{H}} \la H \ra^{-1/4-}  P_{ac}f(x) \big\| _{L^{\infty}}\les  |t|^{-\f1{2}} \|f\|_{L^1} ; 
   \end{split}
        \end{align}   
and if zero is regular, then 
\begin{align} 
    \begin{split}
   &\big\| \la x \ra ^{-{\f 12}-} \cos(t\sqrt{H})\la H \ra^{-3/4-}  P_{ac} f(x) \big\| _{L^{\infty}}\les |t|^{-\f1{2}} \| \la x \ra ^{{\f 12}+} f\|_{L^1}, \\
   & \big\| \frac{\la x \ra ^{-{\f 12}-}\sin(t\sqrt{H})}{\sqrt{H}} \la H \ra^{-1/4-}  P_{ac}f(x) \big\| _{L^{\infty}}\les  |t|^{-\f1{2}} \| \la x \ra ^{{\f 12}+}f\|_{L^1}.
   \end{split}
        \end{align}

The techniques we present below to obtain Theorem~\ref{main1} for the Schr\"odinger evolution can be easily adapted to the wave evolution to obtain:
\begin{theorem} \label{main3}
Let $|V (x)| \les \la x \ra ^{-2\beta} $ for some $ \beta > {3/2} $. If there is a resonance of first kind at zero, then we have
\begin{align*}
& \big\| \la x \ra ^{-{\f 12}-} \cos(t\sqrt{H})\chi(H)P_{ac} f \big \| _{L^\infty} \leq \frac {C} {|t| (\log|t|)^2 } \|\la x\ra ^{{\f 12}+}f\|_{L^1}\\
&\big \|\la x \ra ^{-{\f 12}-}  \big( \frac{\sin(t\sqrt{H})}{\sqrt{H}}\chi(H)P_{ac} g - {\f 1 t } \widetilde{F} g \big) \big \| _{L^\infty} \leq \frac {C} {|t| (\log|t|)^2 } \|\la x\ra ^{{\f 12}+}f\|_{L^1}
\end{align*}
for $|t|>2$, where $\widetilde{F}(x,y)=c \psi(x)\psi(y)$ were $\psi$ is an s-wave resonance function.
   \end{theorem}
   
Theorem~\ref{main3} is valid only for the low energy part but  by including regularizing powers and combining it with the high energy result \eqref{lem:greenwave} of Green, we can extend it to all energies. 
\section{Scalar case} \label{sec:scalar}
In this section we prove that
\begin{theorem} \label{thm:mainineq}
	Under the assumptions of Theorem~\ref{main1}, we have
	for $t>2$
	\be\label{weighteddecay}
		 \bigg| \int_0^\infty e^{it\lambda^2}\lambda \chi(\lambda )
		 [R_V^+(\lambda^2)-R_V^-(\lambda^2)](x,y) d\lambda- {\f 1 t}F(x,y) \bigg|
		 \les \f{\sqrt{w(x)w(y)}}{t\log^2(t)}+
		\f{\la x\ra^{\f32} \la y\ra^{\f32}}{t^{1+\alpha}}
	\ee
	where $0<\alpha<\min(\f14, \beta-\f32)$ and $F(x,y)=- \f{ \psi(x)\psi(y)}{4c_0^2} $ where $\psi $ is the generator of the space of s-wave resonance.
	\end{theorem}
	Here the constant $c_0$  depends on the resonance function $\psi$. Our analysis below determines $c_0$ explicitly. (see, Remark~3 below)
	
Combining \eqref{weighteddecay} with the high energy result obtained in 
\cite{EGw}:
 \begin{theorem}[Theorem 5.1 in \cite{EGw}]  Let $|V (x)| \les \la x \ra ^{-2\beta} $ for some $ \beta > {3/2} $. We have
 $$ \sup_{L\geq 1} \bigg| \int_0^\infty e^{it\lambda^2}\lambda \widetilde{\chi}(\lambda) \chi(\lambda/L)
		 [R_V^+(\lambda^2)-R_V^-(\lambda^2)](x,y) d\lambda \bigg| \les
		\f{\la x\ra^{\f32} \la y\ra^{\f32}}{t^{3/2}}, $$
for $|t|>2$ where $w(x)=\log^2(2+|x|)$.
 \end{theorem} 
 we obtain 
 $$
\big| e^{itH}P_{ac}(H)(x,y) -\frac1t F(x,y)\big|\les \f{\sqrt{w(x)w(y)}}{t\log^2(t)}+
		\f{\la x\ra^{\f32} \la y\ra^{\f32}}{t^{1+\alpha}}.
 $$ 
Interpolating this with 
  \be\nn
	\big| e^{itH}P_{ac}(H)(x,y)  \big| \les \frac{1}{t}
\ee 
from  \cite{EG} which is satisfied when there is a resonance of the first kind at zero and using the inequality (see, e.g., \cite{EGw}):
$$\min\big(1,\f{a}{b}\big)\les \f{\log^2(a)}{\log^2(b)},\,\,\,\,\,\,\,a,b>2, $$	
we  obtain 
\begin{align*}
    \big| e^{itH}P_{ac}(H)(x,y) -\frac1t F(x,y)\big|\les \f{w(x)w(y)}{t \log^2(t)},\,\,\,\,\,\,\,\,t>2.
	\end{align*}
This implies Theorem~\ref{main1}.

\subsection{The Free Resolvent and Resolvent expansion around zero when there is a resonance of the first kind at zero} \label{sec:exp}  
This subsection is devoted to obtain an expansion for the spectral density $[R_V^+(\lambda^2)-R_V^-(\lambda^2)](x,y)$. Recall that in $\R^n$ the integral kernel of the free resolvent is given by Hankel functions \cite{JN}.

For $n=2$ 
\begin{align}\label{R0 def}
\begin{split}
	R_0^\pm(\lambda^2)(x,y)=\pm {\f i 4 } H_0^{\pm}( \lambda |x-y|) =\pm {\f i 4 } \big[J_0(\lambda|x-y| ) \pm i Y_0(\lambda |x-y| ) \big] 
	     \end{split}
	     \end{align}
Here $J_0(z)$ and $Y_0(z)$ are  Bessel functions of the first and second kind of order zero. We use the notation $f=\widetilde{O}(g) $ to indicate 
  \begin{align} 
  \label{frak}
   \frac{d^j}{d\lambda ^j} f = O(\frac{d^j}{d\lambda^j} g ), \hspace{4mm} j=0,1,2,....,. 
     \end{align}
If (\ref{frak}) is satisfied only for $j=1,2,3,..,k$ we use the notation $f=\widetilde{O}_k(g) $.

 For $|z|\ll1$, we have the series expansions for Bessel functions, (see, e.g., \cite{AS,EG}),    
\begin{align}
 	J_0(z)&=1-\frac{1}{4}z^2+\frac{1}{64}z^4+\widetilde O_6(z^6),\label{J0 def}\\
 	Y_0(z)&=\frac2\pi \log(z/2)+\f{2\gamma}{\pi}+ \widetilde O(z^2\log(z)) \label{Y0 def}.
               \end{align} 
For any $\mathcal C\in\{J_0, Y_0\}$ we also have the following representation if $ |z| \gtrsim1 $.
     \begin{align}
     \label{largr}
      	\mathcal C(z)=e^{iz} \omega_+(z)+e^{-iz}\omega_-(z), \qquad
     	\omega_{\pm}(z)=\widetilde O\big((1+|z|)^{-\frac{1}{2}}\big)
          \end{align}
          
 We prove two lemmas on the behavior of $R^{\pm}_0(\lambda^2)(x,y)$ for sufficiently small $\lambda$.\   
\begin{lemma} 
\label{YJ}
Let  $\chi$ be a smooth cutoff for $[-1,1]$, and $\widetilde \chi=1-\chi$. Define $\widetilde{J_0}(z)= \widetilde{\chi}(z)J_0(z).$ Then
 $$|\widetilde J_0(\lambda |z|)| \les \lambda ^{1/2}|z|^{1/2}, \quad |\partial_\lambda \widetilde{J}_0(\lambda |z|)|\les \lambda ^{-1/2}|z|^{1/2}, \quad |\partial_\lambda^2 \widetilde{J}_0(\lambda|z|)|\les \lambda ^{-1/2}|z|^{3/2} .$$
The same bound is satisfied when $J_0(\lambda |z|)$ is replaced with $Y_0(\lambda |z|)$ or $H_0(\lambda |z|)$.
  \end{lemma}
\begin{proof} 
Using (\ref{largr}) we have
   \begin{align*}
   \widetilde{J_0}(\lambda z)&=\Big|\widetilde O\Big(\f{e^{i\lambda|z|}}{(1+\lambda|z|)^{1/2}}\Big)\Big| \les |\lambda z|^{0+} \les \lambda ^{1/2} |z|^{1/2} \\
  |\partial_\lambda \widetilde{J_0}(\lambda z)|&=\widetilde {O}\Big(\dfrac{ze^{i\lambda z}}{(1+\lambda z)^{1/2}} +\dfrac{ze^{i\lambda z}}{(1+\lambda z)^{3/2}}\Big) \les \lambda ^{-1/2}|z|^{1/2}\big[e^{i\lambda z}+\dfrac{e^{i\lambda z}}{\lambda z}] \les \lambda ^{-1/2} |z|^{1/2} \\
   |\partial_\lambda^2 \widetilde{J_0}(\lambda z)|&=\widetilde O\Big(\dfrac{z^2e^{i\lambda z}}{(1+\lambda z)^{1/2}} +\dfrac{z^2e^{i\lambda z}}{(1+\lambda z)^{3/2}}+\dfrac{z^2e^{i\lambda z}}{(1+\lambda z)^{5/2}}\Big) \les \lambda ^{-1/2}|z|^{3/2} 
    \end{align*}
\end{proof}
Define
	\begin{align}
		G_0f(x)&:=-\frac{1}{2\pi}\int_{\R^2} \log|x-y|f(y)\,dy, \label{G0 def}\\
	\label{g form}
			g^{\pm}(\lambda)&:= \|V\|_1\Big(\pm \frac{i}{4}-\frac{1}{2\pi}\log(\lambda/2)-\frac{\gamma}{2\pi} \Big).
	\end{align}
The following lemma and its corollary are Lemma 3.1 and Corollary 3.2 in \cite{EGw}.
\begin{lemma}\label{R0 exp cor}	
	The following expansion is valid for the kernel of the free resolvent
	\begin{align*}
		R_0^{\pm}(\lambda^2)(x,y)=\frac{1}{\|V\|_1} g^{\pm}(\lambda)+G_0(x,y)
		+E_0^{\pm}(\lambda)(x,y).
	\end{align*}
		$G_0(x,y)$ is the kernel of the operator $G_0$ in \eqref{G0 def}, and $E_0^{\pm}$ satisfies the bounds
	\begin{align*}
		|E_0^{\pm}|\les \lambda^{\frac{1}{2} }|x-y|^{\frac{1}{2} }, \qquad
		|\partial_\lambda E_0^{\pm}|\les \lambda^{-\frac{1}{2} }|x-y|^{\frac{1}{2} }, \qquad|\partial_\lambda^2 E_0^{\pm}|\les \lambda^{-\frac{1}{2} }|x-y|^{\frac{3}{2}}.
	\end{align*}
\end{lemma}
\begin{corollary} \label{lipbound} For $0<\alpha<1$  and $b>a>0$ we have
$$
|\partial_\lambda E_0^\pm(b)-\partial_\lambda E_0^\pm(a)|\les a^{-\f12} |b-a|^{\alpha} |x-y|^{\frac12+\alpha}.
$$
\end{corollary}

Define $U(x)$ as $U(x)=1$ when $V(x)>0$ and $U(x)=-1$ when $V(x) \leq 0$, and $v(x)=|V(x)|^{1/2}$. Then using the symmetric resolvent identity for  $\Im\lambda>0$, we have
\begin{align}  \label{rv} 
  R_V ^ {\pm}(\lambda ^2) 
 =R_0 ^ {\pm}(\lambda^2) - R_0 ^ {\pm}(\lambda^2)v M^{\pm}(\lambda)^{-1}vR_0 ^ {\pm}(\lambda^2), 
\end{align}
 where 
\begin{align}
 M^{\pm}(\lambda) = U + vR_0 ^ {\pm}(\lambda^2)v. 
\end{align}
Here we  derive an expansion for $M^{\pm}(\lambda)^{-1}$ in a small neighborhood of zero when there is a resonance of the first kind at zero. This derivation is similar to that in \cite{EGw}. However, we need finer control on the error term.

 Let $K: L^2(\R^n) \rightarrow L^2(\R^n) $ with kernel $K(x,y)$. We define the Hilbert-Schmidt norm of K as

$$\|K\|_{HS} := \sqrt{\int_{\R^n} \int_{\R^n} |K(x,y)|^2 dy dx }.$$

\begin{lemma} \label{lem:M_exp} Let $0<\alpha<1$.
	For $\lambda>0$ define $M^\pm(\lambda):=U+vR_0^\pm(\lambda^2)v$. Then
	\begin{align*}
		M^{\pm}(\lambda)=g^{\pm}(\lambda)P+T+E_1^{\pm}(\lambda).
	\end{align*}
	Here
$T=U+vG_0v$ where $G_0$ is an
	integral operator defined in \eqref{G0 def} and P is the orthogonal projection onto $v$.
In addition, the error term satisfies the bound
	\begin{multline*}
		\big\| \sup_{0<\lambda<\lambda_1} \lambda^{-\frac{1}{2}} |E_1^{\pm}(\lambda)|\big\|_{HS}
		+\big\| \sup_{0<\lambda<\lambda_1} \lambda^{\frac{1}{2}} |\partial_\lambda E_1^{\pm}(\lambda)|\big\|_{HS}	
		\\+\big\| \sup_{0<\lambda<b<\lambda_1} \lambda^{\frac{1}{2}} (b-\lambda)^{-\alpha} |\partial_\lambda E_1^{\pm}(b)-\partial_\lambda E_1^\pm(\lambda)|\big\|_{HS}	
		\les 1
	\end{multline*}
	provided that $v(x)\lesssim \langle x\rangle^{-\frac{3}{2}-\alpha-}$.
	
\end{lemma}

\begin{proof} Note that
$$
		E_1^{\pm}(\lambda) =M^{\pm}(\lambda)-[g^{\pm}(\lambda)P+T]=  vR_0^{\pm}(\lambda^2)v - g^{\pm}(\lambda)P-vG_0v =vE_0^\pm(\lambda)v.
$$
Lemma~\ref{R0 exp cor} and Corollary~\ref{lipbound} yield the lemma since $v(x)|x-y|^k v(y)$ is Hilbert-Schmidt on $L^{2}(\R^2)$ provided that $v(x)\les \la x\ra^{-k-1-}$. In our case  $k\leq \f12+\alpha$ and $v(x)\les \la x\ra^{-3/2-\alpha}$.
\end{proof}
The following definitions are from  \cite{Sc2} and \cite{JN} respectively,

\begin{defin}
We say that an operator $T:L^2(\R^2)\to L^2(\R^2)$ with kernel
	$T(\cdot,\cdot)$ is absolutely bounded if the operator with kernel
	$|T(\cdot,\cdot)|$ is bounded from $L^2(\R^2)$ to $L^2(\R^2)$.
\end{defin}

Hilbert-Schmidt operators and finite rank operators are absolutely bounded. 

\begin{defin} \label{swave}
(1)\ Let $Q:=\mathbbm{1}-P$, then zero is defined to be a regular point of the spectrum of $ H= -\Delta +V $ if $QTQ = Q(U+vG_0v)Q$ is invertible on $QL^2(\R^2)$.\\
(2)\ If zero is not a regular point of spectrum then $QTQ+S_1$ is invertible on $QL^2(\R^2)$ and we define $D_0=(QTQ+S_1)^{-1}$ as an operator on $QL^2(\R^2)$. Here $S_1$ is defined as the Riesz projection onto the Kernel of $QTQ$ as an operator on $QL^2(\R^2)$. \\
(3)\ We say there is a resonance of the first kind at zero if the operator $T_1:= S_1TPTS_1$ is invertible on $S_1L^2(\R^2)$ and we define $D_1$ as the inverse of $T_1$ as an operator on $S_1L^2$.
  \end{defin}
\noindent
{\bf Remarks.}  
\begin{enumerate}
\item Throughout this paper we assume there is a resonance of the first kind at zero. Thus, $QTQ$ is not invertible on $QL^2 $ but $QTQ+S_1$ and $T_1:= S_1TPTS_1$ are invertible on $QL^2 $ and $S_1 L^2$ respectively. 
 
\item By Jensen and Nenciu, we know that Range$(S_1 - S_2)$, $S_2$  being the orthogonal projection on Ker $ T_1$  has dimension at most one [Theorem 6.2 in \cite{JN}]. Since in our case $S_2 \equiv 0$, Range $S_1$ has dimension one and we can write $S_1f=\phi \la f, \phi\ra$ for some $\phi\in S_1L^2$ with $\|\phi\|_{L^2}=1$.
\item Again by Jensen and Nenciu for $v \les \la x \ra ^{-1-} $,  we have $\phi= \omega \psi $ for   an s-wave resonance $\psi\in L^\infty \backslash \big(\cup_{p<\infty} L^p\big)$  such that   $H\psi=0$ in the sense of distributions,  and 
\begin{align}\label{j1}
  \psi = c_0 + G_0v\phi,
    \end{align}
where   
\begin{align} \label{j2}
  c_0 = \frac1{\|V\|_{L^1}} \la v, T\phi\ra = \frac{1}{\|V\| _1}\int v(x) \  T\phi(x) dx.
    \end{align}   
    \item In light of Remark (3)  we have 
    $$ T_1=S_1PTPS_1=\|V\|_1 c_0^2  S_1,\,\,\,\,\,\,  D_1= T_1^{-1}= \f 1{\|V\|_1 c_0^2} S_1.$$ 
\end{enumerate}
The following lemmas are given without proofs.          
 \begin{lemma}  [Lemma 2.1 in \cite{JN}] \label{closed} Let $A$ be closed operator on a Hilbert space  $\mathcal{H}$ and $S$ a projection. Assume $A+S$ has a bounded inverse. Then A has bounded inverse if and only if $ B:= S-S(A+S)^{-1}S$ has a bounded inverse in               $S\mathcal{H}$ and in this case

 $$ A^{-1}=(A+S)^{-1}+(A+S)^{-1}SB^{-1}S(A+S)^{-1} . $$
 
\end{lemma}
\begin{lemma}[Lemma 2.5 in \cite{EG}] \label{matrix}
 Suppose that zero is not a regular point of the spectrum of $-\Delta+V$, and let $S_1$ be the corresponding Riesz projection. Then for sufficiently small $ \lambda_1 > 0 $, the operators $M^{\pm}(\lambda)+S_1$ are invertible for all $0< \lambda < \lambda_1 $ as bounded operators on $L^2(\R^2)$. And one has 

   \begin{align}\label{m+s}
    \big(M^{\pm}(\lambda)+S_1 \big)^{-1}= h_{\pm}(\lambda)^{-1}S +QD_0Q+ W_1^{\pm}(\lambda ),
        \end{align}
   
provided $ v(x)\les \la x\ra ^{-3/2-\alpha-}$. Here $h_{\pm}(\lambda)= g^{\pm}(\lambda)+c $ where $c\in\R$
and
  	\be\label{S_defn}
  	 	 S=\left[\begin{array}{cc} P & -PTQD_0Q\\ -QD_0QTP & QD_0QTPTQD_0Q
		\end{array}\right]
  	\ee
	is a finite-rank operator with real-valued kernel.  Furthermore, the error term satisfies the bound
	\begin{multline*}
		\big\| \sup_{0<\lambda<\lambda_1} \lambda^{-\frac{1}{2} } |W_1{\pm}(\lambda)|\big\|_{HS}
		+\big\| \sup_{0<\lambda<\lambda_1} \lambda^{\frac{1}{2} } |\partial_\lambda W_1{\pm}(\lambda)|\big\|_{HS}	\\
		+\big\| \sup_{0<\lambda<b\les\lambda<\lambda_1}  \lambda^{\frac{1}{2}+\alpha} (b-\lambda)^{-\alpha} |\partial_\lambda W_1{\pm}(b)-\partial_\lambda W_1\pm(a)| \big\|_{HS}	
		\les 1.
	\end{multline*}
	\end{lemma}
	\begin{prop} Assuming $ v(x)\les \la x\ra ^{-3/2-\alpha-}$, in the case of resonance of the first kind at zero $ B_{\pm} (\lambda)= S_1-S_1(M^{\pm}(\lambda+S_1)^{-1}S_1$ is invertible on $S_1L^2(\R^2)$ and we have 
	 $$  B^{-1}_{\pm} (\lambda)= -\frac{h_\pm(\lambda)}{c_0^2 \|V\|_1} S_1+a^{\pm}(\lambda)S_1,$$ 
	 where $c_0$ is as in the third remark and 
\begin{multline*} 
	       \sup_{0<\lambda<\lambda_1} \lambda^{-\frac{1}{2}+ } |a{\pm}(\lambda)|
	 		+ \sup_{0<\lambda<\lambda_1} \lambda^{\frac{1}{2}- } |a^{\prime}{\pm}(\lambda)|
	 		\\+ \sup_{0<\lambda<\theta \les\lambda<\lambda_1}  \lambda^{\frac{1}{2}+\alpha-} (\theta-\lambda)^{-\alpha} |a^{\prime}{\pm}(\theta)-a^{\prime}{\pm}(\omega))| 	
	 		\les 1.
	 	\end{multline*}
      \end{prop}
\begin{proof} We apply Lemma~\ref{closed} to obtain (suppressing '$\pm$' notation)
  \begin{multline*}
          B(\lambda)= S_1-S_1\big(h^{-1}(\lambda)S+QD_0Q+W_1(\lambda)\big)S_1 =- h^{-1}(\lambda)S_1SS_1-S_1W_1(\lambda) S_1 \\
          =h^{-1}(\lambda)S_1TPTS_1 -S_1W_1(\lambda)S_1 =-h^{-1}(\lambda)c_0^2\|V\|_1S_1 -S_1W_1(\lambda)S_1.
                    \end{multline*}
The second equality follows from the identity $Q S_1=S_1Q=S_1D_0=D_0S_1=S_1$. The third also uses the identity $PS_1=S_1P=0$ and the definition of $S$. The last equality follows from Remark 4 above.  

Writing $S_1W_1(\lambda)S_1=w(\lambda) S_1$ (where the function $w$ satisfies the error bound of $W_1$), and noting that by definition of s-wave resonance $c_0\neq 0 $, we obtain $-h^{-1}(\lambda)c_0^2\|V\|_1 -w(\lambda)\neq 0$ for sufficiently small $\lambda$. Therefore
\begin{align} \label{b()}
        B(\lambda)^{-1}= \frac1{-h^{-1}(\lambda)c_0^2\|V\|_1 -w(\lambda)} S_1  
        =  -\frac{h(\lambda)}{c_0^2\|V\|_1}S_1+a(\lambda)S_1.
\end{align} 
The bounds on $a(\lambda)$ follows from the definition of $h$ and the bounds on $w$.
\end{proof}
Using (\ref{m+s}) and (\ref{b()}) in Lemma~\ref{closed}, we obtain the following expansion for $M^{\pm}(\lambda)^{-1}$:
      \begin{corollary} \label{cor} Assume that $v(x)\les \la x\ra ^{-3/2-\alpha-} $. For all $0< \lambda < \lambda_1 $, we have the following expansion for $M^{\pm}(\lambda)^{-1} $ in case of a resonance of the first kind: 
     \begin{multline*}
             M^{\pm}(\lambda)^{-1} = -\frac{h_{\pm}(\lambda)S_1}{c_0^2\|V\|_1} -{\f {SS_1} {c_0^2\|V\|_1}}-{\f {S_1S} {c_0^2\|V\|_1}}-
             \f {SS_1S} {c_0^2\|V\|_1h_{\pm}(\lambda)} +QD_0Q+\f {S}{h_{\pm}(\lambda)}
              +E(\lambda)(x,y)
                              \end{multline*}
where  $E(\lambda)(x,y)$ is such that
   \begin{multline}\label{omega}
 	 		\big\| \sup_{0<\lambda<\lambda_1} \lambda^{-\frac{1}{2}+ } |E_{\pm}(\lambda)|\big\|_{HS}
 	 		+\big\| \sup_{0<\lambda<\lambda_1} \lambda^{\frac{1}{2} } |\partial_\lambda E_{\pm}(\lambda)|\big\|_{HS}	\\
 	 		+\big\| \sup_{0<\lambda<b<\lambda_1}  \lambda^{\frac{1}{2}+\alpha} (b-\lambda)^{-\alpha} |\partial_\lambda E_{\pm}(b)-\partial_\lambda E_{\pm}(a)| \big\|_{HS}	
 	 		\les 1.
 	 	\end{multline}
 	 	
 \end{corollary}
 
 Substituting the expansion above for $M^{\pm}(\lambda)^{-1}$ in (\ref{rv}), we obtain  the identity 
\begin{multline} \label{RV}
       R_V^{\pm}(\lambda)=  R^{\pm}_0(\lambda ^2)  + R^{\pm}_0(\lambda ^2) v \big[ \frac{h_{\pm}(\lambda)S_1}{c_0^2 \|V\|_1} +{\f {SS_1} {c_0^2\|V\|_1}}+\f {S_1S} {c_0^2\|V\|_1}+ \f {SS_1S} {c_0^2\|V\|_1h_{\pm}(\lambda)} \\-QD_0Q
            -\f {S}{h_{\pm}(\lambda)}+E^{\pm}(\lambda) \big]  v R^{\pm}_0(\lambda ^2).
                               \end{multline}
                            
\subsection{Proof of the Theorem~\ref{thm:mainineq}} \hspace{10mm} \\
The following proposition takes care of the contribution of the free resolvent in \eqref{RV} to   \eqref{stone}.
\begin{prop}[Proposition 4.3 in \cite{EGw}] \label{freeevol} We have
$$\int_0^\infty e^{it\lambda^2}\lambda \chi(\lambda)  [R_0^+(\lambda^2)-R_0^-(\lambda^2)](x,y)
d\lambda = - \frac1{4t} +O\Big(\f{\la x \ra^{\f32 } \la y\ra^{\f32 }}{t^{\f54}}\Big).
$$
\end{prop}
 
 Below, we obtain similar estimates for each operator included in (\ref{RV}). Simplifying the boundary terms which appear as operators having ${\f 1 t}$ decay gives us Theorem~\ref{thm:mainineq}. 

The following two stationary phase lemmas from \cite{EGw} will be useful for   further calculations.
 \begin{lemma} \label{lem:ibp} For $t>2$, we have
 $$
 \Big|\int_0^\infty e^{it\lambda^2} \lambda \, \mathcal E(\lambda) d\lambda  - \frac{i\mathcal{E}(0)}{2t}\Big| \les \f1t\int_0^{t^{-1/2}}|\mathcal E^\prime(\lambda)| d\lambda+ \Big|\frac{\mathcal{E}^\prime(t^{-1/2})}{t^{3/2}}\Big|
 +\f1{t^2}\int_{t^{-1/2}}^\infty \Big|\Big(\frac{\mathcal E^\prime(\lambda)}{\lambda}\Big)^\prime\Big| d\lambda.
 $$
 \end{lemma}
  \begin{lemma}\label{lem:ibp2} Assume that  $\mathcal E(0)=0$. For $t>2$, we have
\begin{multline*}
\Big|\int_0^\infty e^{it\lambda^2} \lambda \, \mathcal E(\lambda) d\lambda  \Big|
\les \f1t\int_0^{\infty}\frac{|\mathcal E^\prime(\lambda)|}{  (1+\lambda^2 t)} d\lambda
+\f1{t}\int_{t^{-1/2}}^\infty \big|  \mathcal E^\prime(\lambda \sqrt{1+ \pi t^{-1}\lambda^{-2}} )-\mathcal E^\prime(\lambda)  \big| d\lambda.
\end{multline*}
\end{lemma}

 \vspace{5mm}
We  start with the contribution of $h_{\pm}(\lambda)S_1$ from (\ref{RV}) to (\ref{stone}). Recall that 
$$h_{\pm}(\lambda)= g^{\pm}(\lambda)+c= a_1 \log\lambda+a_2 \pm \dfrac{\|V\|_1 i}{4},$$
where $c, a_i \in \R$, then using the definition \eqref{R0 def} of free resolvent, we write
\begin{multline}
\mathcal R_1 := h_{+}(\lambda)R_0^{+}(\lambda^2)(x,x_1) R_0^{+}(\lambda^2)(y_1,y)- h_{-}(\lambda)R_0^{-}(\lambda^2)(x,x_1) R_0^{-}(\lambda^2)(y_1,y)\\  
\label{calc}
      = 2ia \log(\lambda)[Y_0(\lambda p)J_0(\lambda q)+J_0(\lambda p)Y_0(\lambda q)]\\
        +\dfrac{\|V\|_1 i}{32}[J_0(\lambda p)J_0(\lambda q)+Y_0(\lambda p)Y_0(\lambda q)],
      \end{multline}
      where  $p=|x-x_1|$ and  $q=|y-y_1| $. The following proposition takes care of the contribution of $h_{\pm}(\lambda)S_1$ in \eqref{RV} to   \eqref{stone}.
\begin{prop} \label{pprop} For $t>2$ and $0<\alpha < \frac{1}{4}$, if $v(x)\les \la x\ra^{-3/2-}$, then
 we have 
       $$ \Big|\int_{\R^4}\int_0^\infty e^{it\lambda^2}\lambda \chi(\lambda)   \mathcal R_1(\lambda,p,q) [vS_1v](x_1,y_1)
      d\lambda dx_1 dy_1- \frac{1}{t}F_1(x,y)\Big| \les \dfrac{\la x\ra^{\f1{2}+\alpha}\la y\ra^{\f1{2}+\alpha}}{t^{1+ \alpha}}, $$
where  $$ F_1 = -\dfrac{\|V||_1}{16 \pi^2} \int_{\R^4}\log|x-x_1|v(x_1)S_1(x_1,y_1)v(y_1)\log|y-y_1| dx_1 dy_1. $$ 
\end{prop}
We   prove this proposition in a series of lemmas.     
\begin{lemma} 
\label{lem;asin} Under the conditions of Proposition~\ref{pprop}, we have 
 \begin{align}
 \label{SDS1}
      \Bigg| \int_{\R^4}\int_0^{\infty} e^{it\lambda^2}\lambda \chi(\lambda)\log(\lambda)Y_0(\lambda p) [vS_1v](x_1,y_1) J_0(\lambda q)d\lambda dx_1 dy_1 \Bigg| \les  \dfrac{\la x\ra^{\f1{2}+\alpha}\la y\ra^{\f1{2}+\alpha}}{t^{1+ \alpha}},
          \end{align}
    \begin{align}
    \label{SDS1*}
            \Bigg| \int_{\R^4}\int_0^\infty e^{it\lambda^2}\lambda \chi(\lambda)\log(\lambda)J_0(\lambda p)  [vS_1v](x_1,y_1)Y_0(\lambda q)d\lambda dx_1 dy_1 \Bigg|  \les  \dfrac{\la x\ra^{\f1{2}+\alpha}\la y\ra^{\f1{2}+\alpha}}{t^{1+ \alpha}},
                    \end{align}
      \begin{align}
   \label{SDS12}
  \Bigg| \int_{\R^4}\int_0^\infty e^{it\lambda^2}\lambda \chi(\lambda)J_0(\lambda p)  [vS_1v](x_1,y_1) J_0(\lambda q)d\lambda dx_1 dy_1 \Bigg|  \les   \dfrac{\la x\ra^{\f1{2}+\alpha}\la y\ra^{\f1{2}+\alpha}}{t^{1+ \alpha}}.
          \end{align}               
  \end{lemma}
 To prove Lemma~\ref{lem;asin} we need the following lemma from \cite{EGw}. Another version of it can be found in \cite{Sc2}.
 \begin{lemma}[Lemma 3.3 in \cite{EGw}] \label{FG}
 Let $p=|x-x_1| $, $r=|x|+1 $, and
  $$ F(\lambda , x ,x_1)=\chi(\lambda p)Y_0(\lambda p)-\chi(\lambda  r)Y_0(\lambda r),$$
  $$ G(\lambda , x ,x_1)=\chi(\lambda p)J_0(\lambda p)-\chi(\lambda r)J_0(\lambda r).$$
  We have 
   \begin{align*}
 |G(\lambda , x ,x_1)|\les \lambda ^\frac{1}{2}\la x_1\ra^\frac{1}{2}, \quad |\partial_\lambda G(\lambda ,x,x_1)|\les \lambda ^{-\f 1{2}} \la x_1\ra^\frac{1}{2},\quad |\partial_\lambda^2 G(\lambda ,x,x_1)|\les{\dfrac{1}{\lambda}}\la x_1\ra. 
        \end{align*}
 and       
 \begin{align*} 
  |F(\lambda , x ,x_1)|\les \int_0^{2\lambda_1} |F(\lambda,x,x_1)|+ |\partial_\lambda
  F(\lambda ,x,x_1)| d\lambda \les   k(x,x_1), \quad  \\
   |\partial_\lambda F(\lambda ,x,x_1)|\les \frac{1}{\lambda}, \quad |\partial_\lambda^2 
  F(\lambda ,x,x_1)|\les \dfrac{1}{\lambda^2}. 
\end{align*}
Here $k(x,x_1):=1+\log^-(|x-x_1|)+\log^+(|x_1| + 1 )$,  $\log^-(x)= -\log(x)\chi_{(0,1)}(x)$, and $\log^+(x)= \log(x)\chi_{(1,\infty)}(x)$.
\end{lemma}
\begin{proof}[Proof of Lemma~\ref{lem;asin}] 
We only prove the assertion  (\ref{SDS1}), the second and third assertions are analogous.  

We consider the low and high energy parts separately. To do that we divide the proof into cases,\\
Case $1$: $\lambda p \les 1$ and $\lambda q \les 1 $. Letting $ \chi(\lambda p)$, $ \chi(\lambda q)$ be  the cutoff functions as in Lemma~\ref{YJ}  for low-low energy, we consider   
\begin{align}
\label{cc}
   \int_{\R^4}\int_0^{\infty} e^{it\lambda^2}\lambda \chi(\lambda)\log(\lambda)Y_0(\lambda p) \chi(\lambda p)vS_1v(x_1,y_1)\chi(\lambda q) J_0(\lambda q)d\lambda dx_1 dy_1. 
                              \end{align}
Note that by definition of $S_1$ and $Q$, for any $ f \in L^2(\mathbb{R}^2)$ 
\begin{align}
   \label{zero}
  \int_{\R^4} f(x_1)[S_1](x_1,y_1)v(y_1) dx_1 dy_1 \ =\ \int_{\R^4} 
  v(x_1)[S_1](x_1,y_1)f(y_1) dx_1 dy_1 =0 
\end{align}
is satisfied. By using this fact we can replace $Y_0(\lambda p)\chi(\lambda p)$ with  $F(\lambda , x ,x_1)$; and  $J_0(\lambda q)\chi(\lambda q)$ with $G(\lambda , y ,y_1)$. Hence, the $\lambda $ integral of (\ref{cc}) is equal to
\begin{align}
 \label{SDS2}
\int_0^\infty e^{it\lambda^2}\lambda \chi(\lambda)\log(\lambda)F(\lambda ,x,x_1)G(\lambda ,y,y_1) d\lambda. 
         \end{align}
Letting $\mathcal E(\lambda)= \chi(\lambda)\log(\lambda)F(\lambda ,x,x_1)G(\lambda ,y,y_1)$ , we see that $\mathcal E(0)=0 $. Then by Lemma \ref{FG} we have

 $$ |\partial_\lambda\mathcal E(\lambda)| \les \chi(\lambda) \lambda^{-\f12-}  \la y_1\ra ^{\f12}k(x,x_1), $$
\begin{align*}
      \Big|\partial_\lambda\Big(\frac{\partial_\lambda \mathcal E (\lambda)}{\lambda}\Big)\Big| \les \chi(\lambda)\lambda^{-\f52- } \la y_1\ra k(x,x_1).
 \end{align*}
Using  Lemma~\ref{lem:ibp},
    \begin{multline*}
       |\eqref{SDS2}|
       \les {\f1t}\int_0^{t^{-1/2}}|\mathcal E^\prime(\lambda)| d\lambda+ \Big|\frac{\mathcal{E}^\prime(t^{-1/2})}{t^{3/2}}\Big|
       +{\f1{t^2}}\int_{t^{-1/2}}^\infty \Big|\Big(\frac{\mathcal  E^\prime(\lambda)}{\lambda}\Big)^\prime\Big| d\lambda \\
       \les  \la y_1\ra k(x,x_1)\Bigg[ {\f1t}\int_0^{t^{-1/2}} \lambda^{-\f12-} d\lambda +{t^{-\f54}}
      +\f{1}{t^2}\int_{t^{-1/2}}^\infty  \lambda^{-\f52- }  d\lambda \Bigg] 
      \les \f {k(x,x_1) \la y_1\ra  } {t^{\f54+}}.
   \end{multline*}
 Case 2: $ \lambda p \les 1$ and $ \lambda q \gtrsim 1$. The case $\lambda p \gtrsim 1$ and $ \lambda q \les 1$ is similar. Note that Lemma~\ref{FG} is valid for the low energy. Therefore, we can not use \eqref{zero} to exchange $J_0(\lambda q)\widetilde{\chi}(\lambda q)$ with $G(\lambda , y ,y_1)$. Instead we use the large energy expansion  \eqref{largr} of $J_0(\lambda q)$ and consider the following integral
  \begin{align} 
   \label{SDS3}
                     \int_0^\infty e^{it\lambda^2}\lambda \chi(\lambda)\log(\lambda)F(\lambda ,x,x_1)\widetilde{J_0}(\lambda q) d\lambda.
  \end{align}
Let $\mathcal {E} (\lambda)=\chi(\lambda)\log(\lambda)F(\lambda ,x,x_1)\widetilde{J_0}(\lambda q)$. Using the bounds in Lemma \ref{FG} and Lemma \ref{YJ}, we have the estimates
  \begin{align}
    & |\partial_\lambda\mathcal E(\lambda)| \les \chi(\lambda)\lambda ^{-\f12-}(\la y \ra \la y_1\ra)^{\f12}k(x,x_1), 
          \label{first}\\
     & \big| \partial^2_\lambda \mathcal {E}(\lambda)\big| \les \lambda^{-3/2-} k(x,x_1)\la y_1\ra ^{3/2} \la y \ra ^{3/2}.  \label{second}
           \end{align}
 Using (\ref{second}) for the mean value theorem we have
 $$ \big|\partial_\lambda \mathcal{E}(a) -\partial_\lambda \mathcal {E}(\lambda) \big| \les |a - \lambda|  \lambda^{-3/2-} k(x,x_1)\la y_1\ra ^{3/2} \la y \ra ^{3/2} $$
whose interpolation with (\ref{first}) gives us 
\begin{align}\label{eq;differ}
  \big|\partial_\lambda \mathcal{E}(a) -\partial_\lambda \mathcal {E}(\lambda) \big| \les |a-\lambda |^{\alpha}  \lambda^{-{\f 12}-\alpha} k(x,x_1)\la y_1\ra ^{{\f 12}+\alpha-} \la y \ra ^{{\f 12}+\alpha}.
    \end{align}
Noting $\mathcal{E}(0) =0 $ we can use Lemma~\ref{lem:ibp2} and obtain
\begin{align*}
   |\eqref{SDS3}|\les\f1t\int_0^{\infty}\frac{|\mathcal E^\prime(\lambda)|}{  (1+\lambda^2 t)} d\lambda
+\f1{t}\int_{t^{-1/2}}^\infty \big|  \mathcal E^\prime(\lambda \sqrt{1+ \pi t^{-1}\lambda^{-2}} )-\mathcal E^\prime(\lambda)  \big| d\lambda.
   \end{align*}
Using (\ref{first}) we can estimate the first integral as
$$ \frac{ \la y \ra^{\f12} \la y_1\ra^{\f12}k(x,x_1)}{t} \int_0^{\infty} \frac 1 {\lambda^{-1-}(1+t \lambda^2)} d\lambda \les \frac{\la y \ra^{\f12} \la y_1\ra^{\f12}k(x,x_1)}{t^{5/4+}} $$
To estimate the second integral we have, 
$$  \lambda\big(\sqrt{1+\pi t^{-1}\lambda^{-2}}-1\big) \sim \f 1{t \lambda}.$$
And that gives 
 $$ \f {k(x,x_1)\la y_1\ra ^{{\f 12}+\alpha} \la y \ra ^{{\f 12}+\alpha}}{t^{1+\alpha}} \int_{t^{-1/2}}^{\lambda_1} \lambda^{-{\f1 2}-2\alpha-} d \lambda \les \f {k(,x,x_1)\la y_1\ra ^{{\f 12}+\alpha} \la y \ra ^{{\f 12}+\alpha}}{t^{1+\alpha}} $$
 since $0<\alpha < \f 1{4}$.

\noindent  
Case 3: $\lambda p \gtrsim 1 $ and $\lambda q \gtrsim 1$. In this case we need to use the large energy expansion for both $Y_0(\lambda p)$ and $Y_0(\lambda q)$. Therefore, we consider the following integral 
 \begin{align} \label{eq:case4}
       \int_0^\infty e^{it\lambda^2}\lambda \chi(\lambda)\log(\lambda)\widetilde{Y}_0(\lambda p)\widetilde{J_0}(\lambda q) d\lambda. 
        \end{align}
Note that \eqref{eq:case4} has slightly faster decay than \eqref{SDS3} in terms of $\lambda$. Also the largest contribution to the weight function   comes when both  derivatives act on either $\widetilde{J}_0$ or $\widetilde{Y}_0$  as $\la \cdot\ra^\f 3{2}$. This  can be reduced to $\la \cdot\ra^{\f 1{2} + \alpha} $ using the argument that leads to \eqref{eq;differ} above.   

Hence, combining all four cases we see that
 $$ \Bigg|\int_0^\infty e^{it\lambda^2}\lambda \chi(\lambda)\log(\lambda)Y_0(\lambda p)J_0(\lambda q)d\lambda\Bigg| \les\f{\la y \ra ^{\f 1{2} + \alpha} \la x\ra^{\f 1{2} + \alpha} \la y_1\ra \la x_1\ra }{t^{1+\alpha}} $$
 for $ \alpha \in(0, 1/4 ) $.
That yields
      \begin{align*}
             \big|(\ref{SDS1})\big| &\les \dfrac{(\la x \ra \la y \ra )^{\f 1{2} + \alpha}}{t^{1+\alpha}} \int_{\R^4}k(x,x_1)\la x_1 \ra^{\f 1{2} + \alpha}v(x_1) [S_1](x_1,y_1)v(y_1)k(y,y_1)\la y_1 \ra^{1/2+\alpha} dx_1 dy_1   \\
           &  \les (\la y \ra \la y_1\ra)^{\f 1{2} + \alpha} \| k(x,x_1)\la x_1 \ra ^{1/2+\alpha} v(x_1) \|_{L^2_{x_1}} \| S_1\|_{L^2 \rightarrow L^2}  \| k(y,y_1)\la x_1 \ra  v(y_1) \|_{L^2_{y_1}} \\
            &\les \dfrac{\la x \ra^{\f 1{2} + \alpha} \la y \ra ^{\f 1{2} + \alpha}}{t^{1+\alpha}} .
   \end{align*} 
 The last inequality follows from the assumption     $v(x)\les \la x\ra^{-3/2-}$ which implies
      $$\|v(x_1)\la x_1\ra  k(x,x_1)\|_{L^2_{x_1}}\les 1.$$
  \end{proof}
\begin{lemma}\label{cutoff} Let $K(\lambda,y,y_1) = \chi\big(\lambda |y-y_1| \big) - \chi \big(\lambda (|y|+1)\big)$. Then 
 \begin{align*}
 |K(\lambda , y ,y_1)|\les \lambda \la y_1 \ra  \ ,\  |\partial_\lambda K(\lambda ,y,y_1)|\les \la y_1 \ra \ ,\ |\partial_\lambda^2 K(\lambda ,y,y_1)|\les \lambda ^{-1} \la y_1 \ra . 
 \end{align*}
 \end{lemma}
 \begin{proof}
Noting that $\chi \in C^{\infty}$, for the first inequality we use the mean value theorem to conclude
    \begin{align*}
          |K(\lambda , y ,y_1)|= \big|\chi\big(\lambda q \big) - \chi \big(\lambda (|y|+1)\big) \big| \leq \lambda \la y_1 \ra \max_{x} |\chi^{\prime} (x)| \les \lambda \la y_1 \ra .
       \end{align*}
 For the second inequality note that $ \partial_\lambda \chi( \lambda q) = q \chi^ \prime(\lambda q) $. Using this and the fact that $\chi \in C^{\infty}$, we obtain  
     \begin{align*}
          |\partial_\lambda K(\lambda , y ,y_1)| = \big|q \chi^{\prime}\big(\lambda q \big) - (|y|+1)\chi^{\prime} \big(\lambda (|y|+1)\big) \big| \les  \la y_1 \ra.  
       \end{align*}
  Finally for the third inequality, note that $ \chi^{\prime\prime}(\lambda q) $ is supported when $\lambda \sim \f1{q} $. Using this and the second derivative of the cut-off functions in terms of $\lambda$, we have  
       \begin{align*}
          |\partial_\lambda^2 K(\lambda , y ,y_1)| \leq \big|q^2 \chi^{\prime\prime}\big(\lambda q \big) - (|y|+1)^2\chi^{\prime\prime} \big(\lambda (|y|+1)\big) \big| \les \lambda^{-1} \big|q-(|y|+1)\big|.       
       \end{align*}      
       \end{proof}  
\begin{lemma}\label{lem;S_12} 
Under the same conditions of Proposition \ref{pprop}, we have
\begin{multline}        
  \label{SDS11}
   \int_{\R^4}\int_0^\infty e^{it\lambda^2}\lambda \chi(\lambda)Y_0(\lambda p)vS_1vY_0(\lambda q)d\lambda dx_1 dy_1 \\= - \frac{2i}{t \pi^2} \int_{\R^4} \log|x-x_1| [vS_1v] (x_1,y_1) \log|y-y_1| dx_1 dy_1 + \widetilde{O}\big( t^{-5/4} \la x\ra \la y\ra \big).
                \end{multline}
     \end{lemma}
\begin{proof} The proof is very similar to the proof of Lemma~\ref{lem;asin} except in the case when $\lambda p , \lambda q \les 1 $. This is because the identity
\eqref{zero} leads to an integral with operators $F(\lambda,x,x_1)F(\lambda,y,y_1)$, which  doesn't give better decay rate than $1/t$. We have to be more careful obtaining the term behaving like $1/t$ explicitly. 
 By the expansion  $ Y_0(z)= \dfrac{2}{\pi}\log(z/2)+\dfrac{2\gamma}{\pi}+\widetilde{O}(z^2\log z ) $ of Bessel's function for small energy, we have 
 
       $$Y_0(\lambda p)Y_0(\lambda q) =\f 4 {\pi^2}\log|x-x_1| \log|y-y_1| + A(\lambda,p,q)+ E(\lambda,p,q) $$
where
 \begin{align*}A(\lambda,p,q) := c_1\log(\lambda) [\log(\lambda p)+\log(\lambda q)]+c_2[\log(\lambda p)+\log(\lambda q)] + c_3 \ ,\  where \ c_j \in \R-\{0\} 
     \end{align*}  and  $$E_1(\lambda,p,q) := \widetilde{O}\big(\log(\lambda p)(\lambda q)^2 \log(\lambda q)\big), \quad E_2(\lambda,p,q)=\big( \log(\lambda p)(\lambda p)^2 \log(\lambda q)\big) $$
 To handle the terms in the operator $A(\lambda,p,q)$, we need Lemma~\ref{cutoff}. 
 Consider only the first term in $A(\lambda,p,q)$ then we have
\begin{align} \label{A}
  \int_{\R^4}\int_0^\infty e^{it\lambda^2}\lambda 
  \chi(\lambda)\log\lambda\log(\lambda p) \chi(\lambda p) [vS_1v](x_1,y_1)\chi(\lambda q)d\lambda dx_1 dy_1.
  \end{align}    
Note that by using (\ref{zero}) we can subtract  $\chi(\lambda (|x|+1))$ from the left side of $v(x_1)$ and $\chi \big(\lambda (|y|+1) \big)$ from the right side of $v(y_1)$. Hence (\ref{A}) is controlled by
$$  \int_{\R^4}\int_0^\infty e^{it\lambda^2}\lambda 
  \chi(\lambda)\log \lambda K(\lambda,x,x_1)[vS_1v](x_1,y_1)K(\lambda,y,y_1)d\lambda dx_1 dy_1$$ 
which does not leave any boundry term. Indeed using Lemma \ref{cutoff} and Lemma \ref{lem:ibp} it can be bounded by $  \la y_1 \ra \la x_1 \ra  t^{-3/2+} $. 

For the error term $E_1(\lambda,p,q)$, note that by using  the projection property of $S_1$ we can subtract $\chi(\lambda(|x|+1)\log(\lambda(|x|+1))$ from the left side of the operator $vS_1v$ to replace $\log(\lambda p)$ with $k(x,x_1)$. Then, using $\lambda q \les 1$, we have
\begin{align*}
    & |\partial_{\lambda} [(\lambda q)^2 \log(\lambda q)]| \les q (\lambda q)^{1-} \les \la y \ra \la y_1 \ra, \quad \Big|\partial_{\lambda} \big( \frac{\partial_{\lambda}[(\lambda q)^2 \log(\lambda q)]}{\lambda} \big) \Big| \les \frac{q^2}{\lambda} \les  \frac{ \la y \ra \la y_1 \ra}{\lambda ^2}.
                     \end{align*}
The bound $\frac{k(x,x_1)\la y \ra \la y_1 \ra }{t^{5/4}}$ follows by Lemma~\ref{lem:ibp}. Similarly, the error $E_2(\lambda,p,q)$ can be bounded by $\frac{k(y,y_1)\la x \ra \la x_1 \ra }{t^{5/4}}$.\\
Finally we consider the integral
\begin{align} \label{B}
        \int_{\R^4}\int_0^\infty e^{it\lambda^2}\lambda 
  \chi(\lambda)\log|x-x_1| \chi(\lambda p) [vS_1v](x_1,y_1)\chi(\lambda q) \log|y-y_1|  d\lambda dx_1 dy_1.          
 \end{align} 
Applying integration by parts once, the $\lambda$ integral of (\ref{B}) is equal to
\begin{align}\label{C}
  -\frac{1}{2it}-\frac{1}{4t^2}\int_0^\infty e^{it\lambda^2}\frac{d}{d\lambda}\big(\chi(\lambda)\chi(\lambda p) \chi(\lambda q)\big)d\lambda=-\frac{1}{2it}+O(t^{-5/4} \la x \ra\la x_1 \ra\la y \ra\la y_1 \ra ).
\end{align}
For the second inequality note that all the cut-off functions are infinitely differentiable. However two integration by parts would yield too large of a spatial weight. An easy calculation gives $\Big| \frac{\partial}{\partial \lambda} \big(\chi(\lambda)\chi(\lambda p )\chi(\lambda q) \big)\Big| \les \la x \ra\la x_1 \ra\la y \ra\la y_1 \ra $. And for $ \partial_{\lambda} \big( \frac{\partial_{\lambda}\big(\chi(\lambda)\chi(\lambda p )\chi(\lambda q) \big)}{\lambda} \big)$ the most delicate term comes when all the derivatives fall on either $\chi(\lambda p)$ or $ \chi(\lambda q)$. But since $ \chi ^k  (\lambda p)$ for $k\geq 1$ is supported when $ p \sim {\f 1 \lambda} $ we have 
$$ \bigg| \frac{ \chi(\lambda) \chi \prime \prime (\lambda p) p^2 \chi(\lambda q)}{\lambda} \bigg| \les \lambda^{-5/2} \la x \ra  \la x_1 \ra  $$
and that applying Lemma \ref{lem:ibp} yields (\ref{C}). 

The final result is therefore obtained as
\begin{align*}
  (\ref{SDS11}) &= -\frac{2}{\pi^2 i t} \int _{\R^4} \log|x-x_1| [vS_1 v](x_1,y_1)\log|y-y_1| dy_1 dx_1 \\
   & +O\Big( \dfrac{\la y \ra \la x\ra}{t^{3/2}}   \int_{\R^4}  k(x,x_1)\la x_1 \ra [vS_1 v](x_1,y_1)k(y,y_1)\la y_1 \ra   dx_1 dy_1   \Big), 
         \end{align*}
which finishes the proof of the lemma.
                  \end{proof} 
Multiplying the boundary term with $\dfrac{\|V\|_1 i}{32}$  gives $F_1$ in Proposition~ \ref{pprop}.\\                                
We next consider the contribution of $QD_0Q$, $SS_1$, and  $S_1S$, from (\ref{RV}) to (\ref{stone}). Let 
\begin{multline}\label{r2}
                 \mathcal R_2(\lambda,p,q) :=R_0^{+}(\lambda^2)(x,x_1) R_0^{+}(\lambda^2)(y_1,y) - R_0^{-}(\lambda^2)(x,x_1) R_0^{-}(\lambda^2)(y_1,y)\\ 
=-{\f i 8} \big[ J_0(\lambda p) Y_0(\lambda q) + Y_0(\lambda p) J_0(\lambda q) \big] 
\end{multline}
Note that using this expansion and the projection property of $Q$ the contribution of $QD_0Q$ can be handled as in Proposition \ref{pprop}. Infact, it does not leave any boundary term since \eqref{r2} does not contain the term $Y_0(\lambda p)Y_0(\lambda q) $ 

\begin{prop}\label{propss1}
For $t>2$ and $\alpha \leq \frac{1}{4}$  if $v(x)\les \la x\ra^{-3/2-}$, then we have
$$\bigg|\int_{\R^4}\int_0^\infty e^{it\lambda^2}\lambda \chi(\lambda)  \mathcal R_2(\lambda,p,q) [vSS_1v](x_1,y_1)
      d\lambda dx_1 dy_1 -  \f{1}{t} F_2(x,y)\bigg| \les \dfrac{\la x\ra \la y\ra}{t^{1+ \alpha}}, $$
 $$\bigg|\int_{\R^4}\int_0^\infty e^{it\lambda^2}\lambda \chi(\lambda)  \mathcal R_2(\lambda,p,q) [vS_1Sv](x_1,y_1)
      d\lambda dx_1 dy_1 -  \f{1}{t} F_3(x,y)\bigg| \les \dfrac{\la x\ra \la y\ra }{t^{1+ \alpha}}. $$
where $$ F_2(x,y) = \f{1}{8 \pi } \int_{\R^4} v(x_1)[SS_1](x_1,y_1)v(y_1) \log|y-y_1| dx_1dy_1, $$ 
$$ F_3(x,y) = \f{1}{8 \pi } \int_{\R^4} \log|x-x_1| v(y_1) [S_1S] (x_1,y_1) v(y_1) dy_1 dx_1 . $$ 
\end{prop}
\begin{proof} We consider the first assertion. By \eqref{r2} we have the following two integrals:
\begin{align}    
\label{ss1}
\int_{\R^4}\int_0^\infty e^{it\lambda^2}\lambda \chi(\lambda) Y_0( \lambda p) v(x_1)[SS_1](x_1,y_1)v(y_1) J_0(\lambda q)d\lambda dx_1 dy_1,
\end{align}
\begin{align}
\label{SSDS2}
\int_{\R^4}\int_0^\infty e^{it\lambda^2}\lambda \chi(\lambda)J_0(\lambda p)  v(x_1)[SS_1](x_1,y_1)v(y_1) Y_0(\lambda q)
      d\lambda dx_1 dy_1.
\end{align} 
Here the only caveat is that we have $S_1$ only on the right side, which means that we can perform addition and subtraction of $ J_0(\lambda(|y|+1))$ and $ Y_0(\lambda(|y|+1))$ only on the right side of $SS_1$. Hence the proofs for  high-low and high-high energy are not affected by this caveat. When $ \lambda p\les 1, \lambda q \gtrsim 1 $  we have the following two integrals for (\ref{ss1}) and (\ref{SSDS2}) respectively, 
 \begin{align}
 \label{c1}
        \int_{\R^4}\int_0^\infty e^{it\lambda^2}\lambda \chi(\lambda)\big[ 1+\widetilde{O}\big(\log(\lambda p)\big) \big] \chi(\lambda p) vSS_1v \widetilde{J_0}(\lambda q)
       d\lambda dx_1 dy_1
               \end{align}
  \begin{align} \label{c2}
   \int_{\R^4}\int_0^\infty e^{it\lambda^2}\lambda \chi(\lambda)\big[1+\widetilde{O}\big((\lambda p)^2 \big)] \chi(\lambda p) vSS_1v \widetilde{Y}_0(\lambda q)
        d\lambda dx_1 dy_1 
        \end{align}
 Letting $ \mathcal{E}(\lambda,p,q) = \big[ 1+\widetilde{O}\big(\log(\lambda p)\big) \big] \chi(\lambda p)\widetilde{J_0}(\lambda q)$ we have $\mathcal{E}(0)=0$. Using Lemma \ref{YJ} and the fact that $ (\lambda p) \les 1$, we obtain   
  \begin{align*}
\begin{split}
          & |\partial_\lambda \mathcal{E}(\lambda,p,q)|\les \lambda^{-1/2-} k(x,x_1) \la y \ra^{1/2} \la y_1 \ra^{1/2} \\
          & \big|\partial^2_{\lambda} \mathcal{E}(\lambda,p,q) \big) \big| \les  \lambda^{-3/2-} k(x,x_1) \la x_1 \ra ^{3/2}\la y \ra^{3/2} \la y_1 \ra^{3/2}\\
          & \big| \mathcal{E}(b)-\mathcal{E}(\lambda) \big| \les |b-\lambda|^{\alpha} \lambda^{-3/2-}k(x,x_1) \la x_1 \ra ^{1/2+\alpha}\la y_1 \ra ^{\f1{2}+\alpha} \la y \ra ^{\f1{2}+\alpha}
                \end{split} 
                   \end{align*}  
which gives $$(\ref{c1}) = O\Big( \frac{\la x \ra^{\f1{2}+\alpha} \la y \ra^{\f1{2}+\alpha}}{t^{1+\alpha}} \Big) $$
using Lemma~\ref{lem:ibp2}. With a similar argument  one can show that \eqref{c2} satisfies the same decay assumption with the same weight function. 
                      
For  the low-low case first note that $S_1$ being only on the right side of the operator allows us to exchange $J_0(\lambda q)$ with $G(\lambda,y,y_1$) in \eqref{ss1}, and  $Y_0(\lambda q)$ with $F(\lambda,y,y_1)$ in \eqref{SSDS2}. The decay rate of $G(\lambda,y,y_1$) cancels out the singularity of $\log \lambda $, which is the dominated term in the expansion \eqref{Y0 def} of $Y_0$. Therefore, we don't obtain any boundary term from \eqref{ss1} and can bound it by $\f1{t^{1+\alpha}}$ with the weight $k(x,x_1)\la y \ra \la x \ra $. However, this is not the case for \eqref{SSDS2}. The following lemma evaluates the contribution of this term.
\end{proof}
\begin{lemma}  Under that same conditions of Proposition~\ref{propss1}, for $\lambda p, \lambda q \les 1$ we have 
  \begin{align} \label{eq:ss1boundary} 
  \big|(\ref{SSDS2}) + \dfrac{1}{\pi i t}\int_{\R^4} v(x_1) [SS_1](x_1,y_1)v(y_1) \log|y-y_1| dx_1 dy_1 | \les \frac{ \la x \ra ^{1/2} \la y \ra ^ {1/2} }{t^{5/4}} 
      \end{align} 
\end{lemma}
\begin{proof}
Note that multiplying the boundary term with $ -\f{i}{8} $ gives the the statement of Proposition~\ref{propss1}. 

Using the expansions \eqref{J0 def} and \eqref{Y0 def} for $J_0(\lambda p)$ and $Y_0(\lambda q) $ respectively we have
\begin{align*} 
\begin{split}
   J_0(\lambda p) Y_0(\lambda q) = & \Big[ 1+\widetilde{O}\big((\lambda p)^2\big) \Big] \Big[\frac{2}{\pi}  G_0(y,y_1) + c(1+ \log \lambda) + \widetilde{O} \big( (\lambda q) ^ {2-} \big) \big) \Big] \\
   &=\f 2{\pi} G_0(y,y_1) + c(1+ \log \lambda) + \widetilde{O} \big( (\lambda q)^ {2} \log(\lambda q) \big)  + O_2\big((\lambda p)^2\big) Y_0(\lambda q)
    \end{split}
         \end{align*}
Note that we can exchange $J_0(\lambda p) $ with $ F(\lambda,y,y_1)$ using (\ref{zero}) to obtain
\begin{multline*}
            \hspace{2mm}  (\ref{SSDS2})   = \f 2{\pi} \int_{\R^4}\int_0^\infty e^{it\lambda^2}\lambda \chi(\lambda)\chi(\lambda p))[vSS_1v])(x_1,y_1)\chi(\lambda q) G_0(y,y_1) d\lambda dx_1 dy_1 \\
            +\widetilde{O} \Bigg(  \int_{\R^4}\int_0^\infty e^{it\lambda^2}\lambda \chi(\lambda)\chi(\lambda p) [v SS_1v](x_1,y_1) \chi(\lambda q) (\lambda q)^{2} \log(\lambda q) d\lambda dx_1 dy_1 \\
                                \hspace{5mm} + \int_{\R^4}\int_0^\infty e^{it\lambda^2}\lambda \chi(\lambda)\chi(\lambda p) [v SS_1v](x_1,y_1)\chi(\lambda q)\big[1+\log \lambda\big] d\lambda dx_1 dy_1 \\
                                          + \int_{\R^4}\int_0^\infty e^{it\lambda^2}\lambda \chi(\lambda) \chi(\lambda p) (\lambda p)^2  [v SS_1v](x_1,y_1) F(\lambda ,y, y_1) d\lambda dx_1 dy_1 \Bigg).
                         \end {multline*}
The first integral is similar to \eqref{C}. We therefore have
\begin{align*}
         \f 2{\pi} \int_0^\infty e^{it\lambda^2}\lambda \chi(\lambda)\chi(\lambda p) \chi(\lambda q) G_0(y,y_1) d\lambda =  -\frac{1}{\pi i t}\log|y-y_1|  +O \big( t^{-5/4} \la x \ra\la x_1 \ra\la y \ra\la y_1 \ra \big).
           \end{align*}
The contribution of third integral follows as $A(\lambda,p,q)$ in Lemma~\ref{lem;S_12} and it can be bounded by $ t^{-3/2+} \la x_1 \ra \la y \ra \la y_1 \ra$. Using Lemma~\ref{lem:ibp}, the other two integrals give the same bound that $E_1(\lambda,p,q)$ in Lemma~\ref{lem;S_12} gives. The weights coming from the second derivative of the cut-off functions can be reduced as required using the support of $\chi^\prime(\lambda p)$ and $\chi^\prime(\lambda q)$. Hence, we obtain the inequality \eqref{eq:ss1boundary}.
\end{proof}

For the terms arising from  $h_{\pm}(\lambda)^{-1}SS_1S$ and $h_{\pm}(\lambda)^{-1}S $, which are the integrals 
\begin{align}
\label{ssdss} 
     \int_{\R^4}\int_0^\infty e^{it\lambda^2}\lambda \chi(\lambda)\mathcal{R}_3(\lambda,p,q) vSS_1Sv(x_1,y_1)d\lambda dx_1 dy_1 
          \end{align}
and 
   \begin{align}          
           \label{s}
 \int_{\R^4}\int_0^\infty e^{it\lambda^2}\lambda \chi(\lambda)\mathcal{R}_3(\lambda,p,q)] vSv(x_1,y_1) d\lambda dx_1 dy_1, 
   \end{align}
   where $\mathcal{R}^{\pm}_3=\dfrac{R_0^+(\lambda^2)(x,x_1)R_0^+(\lambda^2)(y,y_1)}{h_{+}(\lambda)}-\dfrac{R_0^+(\lambda^2)(x,x_1)R_0^+(\lambda^2)(y,y_1)}{h_{-}(\lambda)}$ we have the following Proposition, which is the generalized version of Proposition 4.4  in \cite{EGw}.
    \begin{prop}
Let $0< \alpha< 1/4 $, $\textit{v}(x)\les \la x \ra ^ {-3/2-\alpha} $. For any absolutely bounded operator $\Gamma$, we have 
\begin{align*}
   \int_{\R^4}\int_0^\infty & e^{it\lambda^2}\lambda \chi(\lambda) \mathcal{R}_3(\lambda,p,q)\textit{v}(x_1)\Gamma(x_1,y_1)\textit{v}(y_1) d\lambda dx_1dy_1 \\
 &= -\dfrac{1}{4\ \|V\|_1 t}\int_{\R^4}\textit{v}(x)\Gamma(x_1,y_1)\textit{v}(y_1)dx_1dy_1 + O\Big( \f{\sqrt{w(x)w(y)}  }{t\log^2(t)} \Big)+O\Big( \f{\la x\ra^{\f12+\alpha+} \la y\ra^{\f12+\alpha+}}{t^{1+\alpha} } \Big) 
      \end{align*}
\end{prop}

\begin{corollary}\label{corol} Under the same conditions, we have
\begin{align*}
       \big| (\ref{ssdss}) -  {\f 1 t} F_4(x,y) \big| \les O\Big( \f{\sqrt{w(x)w(y)}  }{t\log^2(t)} \Big)+O\Big( \f{\la x\ra^{\f12+\alpha+} \la y\ra^{\f12+\alpha+}}{t^{1+\alpha} } \Big), \\
       \big| (\ref{s}) - {\f 1 t} F_5(x,y) \big| \les O\Big( \f{\sqrt{w(x)w(y)}  }{t\log^2(t)} \Big)+O\Big( \f{\la x\ra^{\f12+\alpha+} \la y\ra^{\f12+\alpha+}}{t^{1+\alpha} } \Big), 
\end{align*}
where 
  \begin{align*}
                       F_4(x,y)= - \dfrac{1}{4\ \|V\|_1 } \int_{\R^2}  v(x_1)[SS_1S](x_1,y_1)v(y_1)  dx_1 dy_1, \\
                           F_5(x,y)= - \dfrac{1}{4\ \|V\|_1 } \int_{\R^2}  v(x_1)[S](x_1,y_1)v(y_1)  dx_1 dy_1 .
       \end{align*}
\end{corollary}

Finally, the contribution of the error term $E(\lambda)(x,y)$ can be handled as in Proposition 4.9 in \cite{EGw}:
\begin{prop} \label{error} Let $0< \alpha< 1/4 $, $\textit{v}(x)\les \la x \ra ^ {-3/2-\alpha} $. We have the bound
\begin{align} 
                  \bigg| \int_{\R^4}\int_0^\infty e^{it\lambda^2}\lambda \chi(\lambda)  [\mathcal R^+_2 -\mathcal R_2^-] v(x_1) E (\lambda) (x_1,y_1) v(y_1)
             d\lambda dx_1 dy_1 \bigg| \les \f{\la x\ra^{\f12+\alpha+} \la y \ra^{\f12+\alpha+} }{t^{1+\alpha}}
                       \end{align}
\end{prop}
 
Using  Proposition~\ref{freeevol}, Proposition~\ref{pprop}, Proposition~\ref{propss1}, Corollary~\ref{corol}, and Proposition~\ref{error}  in the expansion  \eqref{rv}  for $R^+_V -  R^{-}_V $ leads us to \eqref{weighteddecay} with   
  $$    F(x,y)=-\f1{4} +  \f 1{c_0^2\|V\|_1} \sum_{i=1}^4 F_i - F_5   .$$
The next proposition calculates $F(x,y)$ explicitly to finish the proof of Theorem~\ref{thm:mainineq}:
\begin{prop} \label{psipsi} Under the conditions of Theorem~\ref{main1},
\begin{align} \label{eq;Ftop} 
F(x,y) = - \f{1}{4c_0^2} \psi(x) \psi(y) 
   \end{align}
 where $\psi$ is an s-wave resonance. 
          \end {prop}        
\begin{proof} Recall  that  $S_1$ is a projection operator with the kernel $S_1(x,y)= \phi(x)\phi(y) $ for some $\|\phi\|_{L^2}=1$. Using this and the definition \eqref{g form} of $G_0f(x)$, $F_1$ can be written as
\begin{align*}
                          F_1( x,y)  
    =  -\frac{ \|V\|_1 } { 4} [G_0v\phi](x) [G_0v\phi](y).
                         \end{align*}
By Remark 3 we know that $G_0v\phi= \psi -c_0$, where $\psi$ is a resonance function. This gives us  
  $$ F_1(x,y) =-\frac{ \|V\|_1 } { 4}\big( \psi(x)-c_0 \big) \big(\psi(y)-c_0\big).  $$       
 
For $F_2$ and $F_3$ recall that 

  	$$ 	 S=\left[\begin{array}{cc} P & -PTQD_0Q\\ -QD_0QTP & QD_0QTPTQD_0Q 
		\end{array}\right] = \left[\begin{array}{cc} a_{11} & a_{12} \\ a_{21} & a_{22}
		\end{array}\right].$$
		
Note that multiplying $S$ by $v$ from the left side cancels $a_{21}$ and  $a_{22}$; and by $S_1$ from the right side cancels $a_{11}$. Hence, we see that $vSS_1= -vPTQD_0Q S_1 = -vTS_1$. Here we also used the fact that $S_1D_0=D_0S_1= S_1$.  Therefore, we have
                 \begin{multline*} 
                 F_2( x,y) = \f{1}{8 \pi } \int_{\R^4} v(x_1)[SS_1](x_1,y_1) v(y_1) \log|y-y_1|  dy_1dx_1\\
                 =\f{1}{8 \pi } \int_{\R^4} v(x_1)TS_1(x_1,y_1) v(y_1) \log|y-y_1|  dy_1dx_1
                                                             \end{multline*}
which is equal the following by using the definition of $S_1$ and $G_0f(x)$
\begin{multline*}
 \f{1}{8 \pi } \int_{\R^4} v(x_1)[T\phi](x_1) [\phi v](y_1) \log|y-y_1| dx_1 dy_1 = -\f1{4} \la v , T\phi \ra [G_0v\phi](y) \\
 = -\frac{ \|V\|_1 } { 4} c_0 \big( \psi(y) - c_0 \big).
             \end{multline*}                                                           
For the last equality we again used Remark 3. The same calculation shows that 
 \begin{align*}
                F_3( x,y) =  -\frac{ \|V\|_1 } { 4} c_0 \big( \psi(x) - c_0 \big).  
                                           \end{align*}
                                           
Using the above definition of $S$ and the same cancellations, we have $ vSS_1Sv=vTS_1Tv $ which results in                                       
\begin{align*}
                          F_4( x,y) &=  -\f 1{4 \|V\|_1}  \la v , T\phi \ra \la v , T\phi \ra= -  \frac{ \|V\|_1 } {4} \ c_0^2.   
                                            \end{align*} 
For   $ F_5( x,y) $, note that we have  $v(x)$ both on left and right side of S. Hence,  except $ P$ everything vanishes and we obtain
\begin{align*}
                          F_5( x,y) =  -{\f 1 {4 \|V\|_1}}  \int _{\R^2}  v(x_1) P(x_1,y_1) v(y_1) dx_1 dy_1 
                           = - {\f 1 4 }. 
                                            \end{align*} 
It is easy to see that $F_5$ cancels out the operator coming from the free resolvent. The other four sum up to $ -\frac{ \|V\|_1 } { 4}\psi(x) \psi(y) $ and that establishes the proof.
                                            
\end{proof}

We conclude this section by remarking that the bounds that we obtain in this section allows us to reach a similar estimate for the solution of the wave equation with some small modifications. Replacing Proposition~\ref{freeevol}, Proposition~\ref{propss1}, and Proposition~\ref{error} with Proposition~5.10, Proposition~5.11, and Proposition~5.15 in \cite{greenw} respectively one can obtain: 
\begin{multline*}
\Big| \int_0^\infty \big( \sin(t\lambda)+ \lambda \cos(t\lambda) \chi(\lambda) \big)[R_V^+(\lambda^2) - R_V^{-}(\lambda^2)](x,y) d\lambda - {\f 1 t} \widetilde{F}(x,y) \Big| \\
     \les \frac{(1+\log^+|x|)(1+\log^+|y|)}{t\log^2 t} + \frac{\la x\ra ^{{\f 12} +\alpha}\la y \ra ^{{\f 12} +\alpha}}{t^{1+\alpha}} .
 \end{multline*}
 
 This estimate gives us Theorem~\ref{main3} with no interpolation. Note that the interpolation with unweighted result (14) does not help us to decrease the weight function to $\log^2(2+|x|)$ and have the decay $(t \log^2t)^{-1}$. This is because we need to improve the time decay from $|t|^{-{1/2}}$ as opposed to Schr\"odinger time decay $|t|^{-1}$ .
  
Also note that we only need to subtract a finite rank operator from \eqref{stonesin}. The reason is the following identities ($\Lambda$ smooth and compactly supported)
$$
\int_0^\infty \cos(t\lambda)\lambda \Lambda(\lambda) d\lambda =-\frac1t\int_0^\infty \sin(t\lambda) \big(\lambda \Lambda(\lambda)\big)^\prime d\lambda,
$$
$$
\int_0^\infty \sin(t\lambda) \Lambda(\lambda) d\lambda = -\frac1t\Lambda(0)+\frac1t\int_0^\infty \cos(t\lambda) \Lambda^\prime(\lambda)d\lambda.
$$ 
The boundary term in the second identity will result in the finite rank operator, as in the proof of Theorem~\ref{main1}.
 
\section{MATRIX CASE } \label{matrixx}

We start to the matrix case by reminding that the following representation is valid for $(f,g) \in W^{2,2} \times W^{2,2} \cup X_{1+} $,(see, section~2 in \cite{ES2});
\begin{align} \label{pacrep} 
                \la e^{it\mathcal{H}}P_{ac}f,g  \ra = {\f 1 {2 \pi i}} \int_{|\lambda| > \mu} e^{it\lambda} \la \big[ \mathfrak{R}^+_V (\lambda) -  \mathfrak{R}^-_V (\lambda) \big] f,g \ra d\lambda.
            \end{align} 
           
 We prove the following two theorems and use the interpolation argument from the scalar case.
\begin{theorem} \label{t21}Under the assumptions of A1)-A4), if there is a resonance of the first kind at zero then we have, for any $ t \geq 0$,
\begin{align*} 
           \sup_{{x,y \in \R^2}, L>1} \Big| \int_{|\lambda|>\mu } e^{it\lambda} \chi(\lambda/L) \big[ \mathfrak{R}^+_V (\lambda) -  \mathfrak{R}^-_V (\lambda) \big] (x,y) d\lambda \Big| \les {\f 1 {|t|}} .
           \end{align*}
      \end{theorem}
\begin{theorem} \label{t22}Under the assumptions A1)-A4), if there is a resonance of the first kind at zero then we have, for any $ t>2 $,   
\begin{align*}      
            \sup_{ L>1} \Big| \int_{|\lambda|>\mu } e^{it\lambda} \chi(\lambda/L) \big[ \mathfrak{R}^+_V (\lambda) -  \mathfrak{R}^-_V (\lambda) \big] (x,y) d\lambda  - {\f 1 t} \mathfrak{F}
            (x,y) \Big| \les   \f{\sqrt{w(x)w(y)}  }{t\log^2(t)} + \f{\la x\ra^{3/2} \la y\ra^{3/2}}{t^{1+\alpha} }             
                \end{align*}
where $ 0<\alpha< {\f{\beta -3} 2 } $.               
          \end{theorem}
  
\subsection{The free resolvent and resolvent expansion around zero in case of s-wave resonance}\hspace{10mm}\\
The aim of this part of the section is to show the spectral density $[\mathfrak{R}^{+}_V (\lambda)-\mathfrak{R}^{-}_V (\lambda)](x,y)$ has a similar expansion as in the scalar case. The free resolvent $ \mathfrak{R}_0 (z) $ of matrix Schrodinger equation is given by
$$ \mathfrak{R}_0 (z) = (\mathcal{H}_0 - z)^{-1} = \left[\begin{array}{cc} R_0(z-\mu) & 0 \\ 0 & -R_0(-z-\mu) \end{array}\right] $$
for $ z \in (-\infty , - \mu) \cup (\mu , \infty) $. Here $R_0 ( z) $ is the scalar free resolvent.
Writing $z= \mu + \lambda ^2 $ , $ \lambda >0 $ we have
$$ \mathfrak{R}_0 (\mu + \lambda ^2 )(x,y)= \left[\begin{array}{cc} R_0(\lambda ^2 )(x,y) & 0 \\ 0 & -{\f i 4 }H^+_0(i \sqrt{2\mu + \lambda ^2 }|x-y| )  \end{array}\right] .$$
Note that the bounds
\begin{align}\label{r2 bound}
     |R_2(\lambda ^2) (x,y)| \les 1 + \log^{-}|x-y| \les k(x,y)
     \hspace{1mm}, \hspace{1mm} |\partial_\lambda^k R_2(\lambda^2)(x,y)| \les 1 \hspace{5mm} k=1,2,...
     \end{align}
 can be seen directly from the large and small energy expansion of Hankel functions and $\mu$ being strictly greater than zero.
     
We will repeat some Lemmas and Corollaries from Section~\ref{sec:scalar} modified as needed for the matrix operator.\\
Define the matrices
$$M_{11}= \left[\begin{array}{cc} 1 & 0 \\ 0 & 0 \end{array} \right] , \hspace{10mm} M_{22}= \left[\begin{array}{cc} 0 & 0 \\ 0 & 1 \end{array} \right] .$$

\begin{lemma} \label{mR0 exp cor}The following expansion is valid for the kernel of the free resolvent 
                     $$ \mathfrak{R}_0 (\mu + \lambda ^2 )(x,y)= \mathfrak{g}^{\pm}(\lambda) M_{11}+ \mathcal{G}_0(x,y)+ \mathcal{E}_0^{\pm} (\lambda)(x,y), $$                           
where 
\begin{align*} 
              \mathfrak{g}^{\pm}(\lambda) = \pm \frac{i}{4}-\frac{1}{2\pi}\log(\lambda/2)-\frac{\gamma}{2\pi}, \\
  \vspace{10pt}
              \mathcal{G}_0(x,y)=   \left[\begin{array}{cc} G_0(x,y) & 0 \\ 0 & -{\f i 4} H_0^+(i \sqrt{2\mu}|x-y|) \end{array} \right],   
                     \end{align*}                                                    
                                            
 and  $\mathcal{E}_0^{\pm} (\lambda)(x,y) $ satisfies the bounds,
 \begin{align*}
		|\mathcal{E}_0^{\pm}|\les \la \lambda\ra^{\f 1 2 } \lambda^{\f 1 2} \la x-y \ra^{\frac{1}{2} }, \qquad
		|\partial_\lambda \mathcal{E}_0^{\pm}|\les \la \lambda\ra^{\f 1 2 } \lambda^{-\frac{1}{2} }\la x-y \ra^{\frac{1}{2} }, \qquad|\partial_\lambda^2 \mathcal{E}_0^{\pm}|\les \la \lambda\ra^{\f 1 2 } \lambda^{-\frac{1}{2} }\la x-y \ra^{\frac{3}{2}}.
	\end{align*}
        \end{lemma}       
\begin{corollary} \label{mlipbound} For $0<\alpha<1$  and $b>a>0$ we have,
$$
|\partial_\lambda \mathcal{E}_0^\pm(b)-\partial_\lambda \mathcal{E}_0^\pm(a)|\les a^{-\f12} |b-a|^{\alpha} \la x-y \ra ^{\frac12+\alpha}.
$$
\end{corollary}
We write $V= -\sigma_3vv:=v_1v_2 $ where  $v_1=  -\sigma_3v $ , $v_2= v $, and 
 $$v={\f 1 2}  \left[\begin{array}{cc} \sqrt{V_1+V_2}+\sqrt{V_1-V_2} &  \sqrt{V_1+V_2}- \sqrt{V_1-V_2} \\\sqrt{V_1+V_2}-\sqrt{V_1-V_2}  & \sqrt{V_1+V_2}+\sqrt{V_1-V_2}  \end{array} \right]: =  \left[\begin{array}{cc} a & b \\ b & a \end{array} \right].                                   $$
Using symmetric resolvent identity, we have
\begin{align*}
          \mathfrak{R}_V (\mu + \lambda ^2 )= \mathfrak{R}_0 (\mu + \lambda ^2 ) - \mathfrak{R}_0 (\mu + \lambda ^2 ) v_1 M^{\pm} (\lambda)^{-1} v_2 \mathfrak{R}_0 (\mu + \lambda ^2 ),
          \end{align*}
 where
 \begin{align*} 
            M^{\pm} (\lambda) = I +  v_2  \mathfrak{R}_0 (\mu + \lambda ^2 ) v_1.
            \end{align*}
Employing Lemma \ref{mR0 exp cor}, 
\begin{align*}
     M^{\pm} (\lambda)=  \mathfrak{g}^{\pm}(\lambda) v_2 M_{11}v_1 + T + v_2 \mathcal{E}_0^{\pm} v_1  
     \end{align*}                              
where T has kernel $T(x,y) = I + v_2(x) \mathcal{G}_0(x,y)v_1(y) $. 
\begin{lemma} Let $ 0< \alpha < 1$. The following expansion is valid for $\lambda > 0 $ 
$$ M^{\pm} (\lambda)= -\|a^2+b^2\|_{L_1(\R^2)}\mathfrak{g}^{\pm}(\lambda) P + T + \mathcal{E}_1^{\pm}(\lambda), $$ 
where $P$ is the orthogonal projection onto the span of the vector $(a,b)^T$ in $L^2 \times L^2 $. Further, we have
 \begin{multline*}
		\big\| \sup_{0<\lambda<\lambda_1} \lambda^{-\frac{1}{2}} |\mathcal{E}_1^{\pm}(\lambda)|\big\|_{HS}
		+\big\| \sup_{0<\lambda<\lambda_1} \lambda^{\frac{1}{2}} |\partial_\lambda \mathcal{E}_1^{\pm}(\lambda)|\big\|_{HS}	
		\\+\big\| \sup_{0<\lambda<b<\lambda_1} \lambda^{\frac{1}{2}} (b-\lambda)^{-\alpha} |\partial_\lambda \mathcal{E}_1^{\pm}(b)-\partial_\lambda \mathcal{E}_1^\pm(\lambda)|\big\|_{HS}	
		\les 1,
	\end{multline*}                                     
                                            
provided that $a(x),b(x) \les \la x \ra ^ { -3/2-\alpha - }$.  
\end{lemma}
\begin{proof}
Note that the formulas for $v_1$ and $v_2$ give us,
 $$ \mathfrak{g}^{\pm} (\lambda) v_2M_{11}v_1=-\mathfrak{g}^{\pm} (\lambda) \left[\begin{array}{cc} a & 0 \\ b & 0 \end{array} \right] \left[\begin{array}{cc} a & b \\ 0 & 0 \end{array} \right] = - \|a^2+b^2\|_{L_1(\R^2)}\mathfrak{g}^{\pm}(\lambda)P. $$ 
 The Hilbert-Schmidt bound comes by the assumption on $a(x),b(x)$, Lemma \ref{mR0 exp cor}, and its corollary.                                    
\end{proof}
 Recall $P$ in the scalar case is defined as projection onto $v$ whereas in matrix case it is defined as projection onto the span of the vector $(a,b)^T$. In light of this difference we will give the following modified version of Definition \ref{swave}. Let $Q:=\mathbbm{1}-P$.
 \begin{defin} \label{swavem}
 (1) $\mu$ is defined to be a regular point of the spectrum of $\mathcal{H}= -\Delta +V $ if $QTQ $ is invertible on $Q(L^2 \times L^2)$.\\
(2) If $\mu$ is not a regular point of spectrum then $QTQ+S_1$ is invertible on $Q(L^2 \times L^2)$ and we define $D_0=(QTQ+S_1)^{-1}$ as an operator on $Q(L^2 \times L^2)$. Here $S_1$ is defined as Riesz projection onto the Kernel of $QTQ$ as an operator on $Q(L^2 \times L^2)$. \\
(3) We say there is a resonance of the first kind at zero if the operator $T_1:= S_1TPTS_1$ is invertible on $S_1 Q(L^2 \times L^2)$ and we define $D_1$ as the inverse of $T_1$.
      \end{defin}
    With the following lemma we can have a representation for the space $S_1$ as in the scalar case.
\begin{lemma} \label{G0matrix} If $|a(x)|+|b(x)| \les \la x \ra ^{-1 -}$ and if $\phi \in S_1(\L^2 \times \L^2 )$ , then  $\phi(x)=v_2 \psi_1 = \psi_2 v_1 $ where $\psi_1,\psi_2 \in L^{\infty} \times L^{\infty} $ and $(\mathcal{H}_0-\mu I )\psi_i = 0 $ for $i=1,2$ in the sense of distribution. Also we have 
$$ \psi_1(x)= - \int_{\R^2} \mathcal{G}_0(x,y) v_1(x) \phi(x)dx + (c_0,0) \hspace{2mm},\hspace{2mm} \psi_2(x)= - \int_{\R^2} \phi(x) v_2(x)\mathcal{G}_0 (x,y) -(c_0,0), $$
with 
$$ c_0 = \frac{1}{ \|a^2+b^2\|_1 } \la T\phi , (a,b) \ra.$$ 
\end{lemma}
\label{mphi}
\begin{proof}We will prove  $\phi(x) = \psi_2 v_1 $. Note that for any $\phi \in S_1(\L^2 \times \L^2) $ since $ S_1 \leq Q $ we have $ \phi Q = \phi $. Also using $ Q=1-P $ we have 
\begin{align*}
  \begin{split}
   0 &= \int_{\R^2} \phi(x) QTQ(x,y) dx  = \int_{\R^2} \phi(x)T(I-P)(x,y) dx \\
    &= \int_{\R^2} \phi(x) [I+v_2\mathcal{G}_0v_1](x,y) dx + \int_{\R^2} \phi(x)[T P](x,y) dx\\
    &= \phi+ \int_{\R^2} \phi(x) v_2(x)\mathcal{G}_0 (x,y) dx \ v_1(y) +(- c_0,0)v_1(y)= \phi + \psi_2v_1
      \end{split}
         \end{align*}
For the second equality we used the definition of $T$. For the third equality we used the definition of $P$ to obtain
 $$  \int_{\R^2} \phi(x)[T P] (x,y) dx =  \frac{\la (a,b)^T, T\phi \ra }{\|a^2+b^2\|_1 }(a,b)= c_0(a,b)=c_0(-1 ,0)v_1(y).$$
 For the proof  of the first part one can see Lemma 4.4 in \cite{EGm}. Indeed, it follows with 
$$ 0= \int_{\R^2} QTQ(x,y)\phi(x) dx \ with \ c_0 \begin{pmatrix} 
                                                    					a\\
											b
											\end{pmatrix} =  c_0 v_2  \begin{pmatrix} 
                                                    					1\\
											0
											\end{pmatrix}  $$
For $\psi $ being in $L^{\infty}$ one can see Lemma 5.1 in \cite{EG}.											
  \end{proof} 
{\bf Remarks.}
\begin{itemize}
    \item  $\phi(x)=v_2 (x)\psi_1(x) = \psi_2(x) v_1(x) $ gives us $\sigma_3 \psi_1  = \psi_2 $.
    \item If there is a resonance of the first kind at zero, Range $S_1$ is one dimensional and if we take $ \|\phi\|_{L^2 \times L^2 }=1$ with $\phi \in S_1(\L^2 \times \L^2 )$ then,  $S_1f= \big(\phi_1,\phi_2 \big) \la \phi , f \ra  $ where $\phi$ is as in the Lemma~.\ref{mphi}.
    \item By Lemma~\ref{mphi}, we have
            $$ D_1= \frac 1{\|a^2+b^2\|_1c_0^2}S_1. $$
         \end{itemize}            
         
Definition~\ref{swavem} and Lemma~\ref{mphi} followed by the steps in scalar case gives us the same expansion for $ M^{\pm}(\lambda)^{-1}$ with $\|a^2+b^2\|_1$ instead of $\|V\|_1$. Hence, for $h_{\pm}(\lambda)= -\|a^2+b^2\|_1\mathfrak{g}^{\pm}(\lambda)+c$ where $c\in \R$ and for $ 0<\lambda < \lambda_1$, we have
 \begin{multline} \label{matrixRV}
      \mathfrak{R}_V^{\pm}(\lambda)=  \mathfrak{R}^{\pm}_0(\lambda ^2)  + \mathfrak{R}^{\pm}_0(\lambda ^2) v_1 \big[ \frac{h_{\pm}}{\|a^2+b^2\|_1c_0^2} (\lambda)S_1+\f {SS_1}{\|a^2+b^2\|_1c_0^2}D\\ +\f {S_1S}{\|a^2+b^2\|_1c_0^2}+\f1{\|a^2+b^2\|_1c_0^2}h_{\pm}^{-1}SS_1S
                -h_{\pm}^{-1}(\lambda)S-QD_0Q-E^{\pm}(\lambda) \big]  v_2 \mathfrak{R}^{\pm}_0(\lambda ^2)
                               \end{multline}                                           
 with  $ E(\lambda)(x,y) $ is such that 
 \begin{multline*}
 	 		\big\| \sup_{0<\lambda<\lambda_1} \lambda^{-\frac{1}{2}+ } |E^{\pm}(\lambda)|\big\|_{HS}
 	 		+\big\| \sup_{0<\lambda<\lambda_1} \lambda^{\frac{1}{2} } |\partial_\lambda E^{\pm}(\lambda)|\big\|_{HS}	\\
 	 		+\big\| \sup_{0<\lambda<b<\lambda_1}  \lambda^{\frac{1}{2}+\alpha} (b-\lambda)^{-\alpha} |\partial_\lambda E^{\pm}(b)-\partial_\lambda E^\pm(a)| \big\|_{HS}	
 	 		\les 1.
 	 	\end{multline*}
 Here the matrix $S$ has the same definition \eqref{S_defn} as in the scalar case.
 \subsection{Proof of the Theorem \ref{t22}}  \hspace{2mm}\\
The proof of Theorem \ref{t22} is  similar to the proof of Theorem~\ref{thm:mainineq}. The cancellation property  $Qv=0$ that we used repeatedly is replaced with 
\begin{align} \label{m0}
      M_{11}v_1S_1= S_1v_2M_{11} = 0,
         \end{align}
which allows us to use Lemma~\ref{FG} to gain extra time decay.   Furthermore, as in the scalar case,  the boundary terms arise only in the low-low energy evolution.   For this reason, we present the proof of Theorem \ref{t22} for the case $\lambda p$, $\lambda q \les 1$, and   omit the cases in which high energy is involved. For high energies one can apply the same methods that we applied in the scalar case using the bound \eqref{r2 bound} in addition to the bound \eqref{largr}, see \cite{EGm} for similar arguments.
 
 For convenience we write 
$$  \mathfrak{R}_0 (\mu + \lambda ^2 )(x,y)= R_0(\lambda ^2 )(x,y) M_{11} + R_2(\lambda ^2 )(x,y) M_{22}.$$

 The following Proposition takes care of the contribution of 
 \begin{align}  \label{r1int} 
     \int_{\R^4} \int_0^\infty e^{it\lambda^2} \lambda \chi(\lambda) \mathfrak{R}_1(\lambda,p,q) [v_1S_1 v_2](x_1,y_1) d\lambda dx_1 dy_1
       \end{align}
to \eqref{pacrep} where 
\begin{multline*}
 \mathfrak{R}_1(\lambda,p,q):= h^{+}(\lambda) \mathfrak{R}^{+}(\mu + \lambda ^2 )(x,x_1) \mathfrak{R}^{+}(\mu + \lambda ^2 )(y,y_1)\\
    - h^{-}(\lambda) \mathfrak{R}^{-}(\mu + \lambda ^2 )(x,x_1) \mathfrak{R}^{-}(\mu + \lambda ^2 )(y,y_1). 
   \end{multline*}     
   \begin{prop} Let $0<\alpha \leq 1/4$. If $|a(x)|+|b(x)| \les \la x \ra ^{-3/2 -}$ then, we have
\begin{align*}
 \Big|  \eqref{r1int}  - {\f 1 t} \mathfrak{F}_1 (x,y) \Big| \les \frac{\la x \ra \la y \ra }{t^{1+\alpha}}, 
                 \end{align*}
 where
\begin{align*}
         \mathfrak{F}_1 (x,y) = \frac{\|a^2+b^2\|_1} {4 } \int_{\R^4}\mathcal{G}_0(x,x_1)v_1(x_1)S_1(x_1,y_1)v_2(y_1) \mathcal{G}_0(y,y_1) dx_1dy_1.
          \end{align*}
\end{prop}
\begin{proof}
$\mathfrak{R}_1(\lambda, p, q)$ can be calculated as
\begin{align*}
  \begin{split} 
              &h^{+}(\lambda)R_0^{+}(\lambda^2)(x,x_1)M_{11}M_{11} R_0^{+}(\lambda^2)(y_1,y)-h^{-}(\lambda)R_0^{-}(\lambda^2)(x,x_1)M_{11}M_{11} R_0^{-}(\lambda^2)(y_1,y)\\
             &\hspace{10mm}+  [h^{+}(\lambda)R_0^{+}(\lambda^2)(x,x_1)-h^{-}(\lambda)R_0^{-}(\lambda^2)(x,x_1)]M_{11}M_{22}R_2(\lambda^2)(y_1,y) \\
              &\hspace{15mm}+R_2(\lambda^2)(x,x_1) M_{22}M_{11}[h^{+}(\lambda)R_0^{+}(\lambda^2)(y_1,y))-h^{-}(\lambda)R_0^{-}(\lambda^2)(y_1,y))]\\
              &\hspace{20mm}+ [h^{+}(\lambda)-h^{-}(\lambda)]M_{22}M_{22} R_2(\lambda^2)(x,x_1)R_2(\lambda^2)(y_1,y)\\
              & \hspace{25mm}=A_1(\lambda,p,q)+A_2(\lambda,p,q)+A_3(\lambda,p,q)+A_4(\lambda,p,q).
                    \end{split}
                                \end{align*}
Note that $A_1(\lambda,p,q)$ is similar to \eqref{calc}. Hence, using the projection property \eqref{m0}, its contribution to the integral  \eqref{r1int} can be obtained as
 \begin{align} \label{op1}
     \frac{\|a^2+b^2\|_1 } {16 \pi^2} \int_{\R^4} G_0(x,x_1)M_{11}v_1S_1v_2M_{11}G_0(y,y_1) dx_1 dy_1  
     + O\Big(\frac{\la x\ra^{\frac{1}{2}+\alpha} \la y\ra^{\frac{1}{2}+\alpha}}{t^{1+\alpha}}\Big). 
                 \end{align}
Next we consider $A_4(\lambda,p,q)$. First note that
 \begin{multline*} 
 [h^{+}(\lambda)-h^{-}(\lambda)]R_2(\lambda^2)(x,x_1)R_2(\lambda^2)(y_1,y) = \\ -\frac{\|a^2+b^2\|_1 i} {32}H_0^+(i \sqrt{2\mu + \lambda^2}\ p)  H_0^+(i \sqrt{2\mu+ \lambda ^2 }\ q).
   \end{multline*}
  Taking $ \mathcal{E}(\lambda,p,q) = \chi(\lambda) H_0^+(i \sqrt{2\mu + \lambda^2}\ p)  H_0^+(i \sqrt{2\mu+ \lambda ^2 }\ q)$ we see that $\mathcal{E}(0) \\ =H_0^+(i  \sqrt{2\mu}|x-x_1|)  H_0^+(i \sqrt{2\mu}|y-y_1| )$. Also the bounds \eqref{r2 bound} leads us to:
\begin{align} \label{firstder} 
       \Big|\frac{\partial}{\partial\lambda} [ \chi(\lambda) H_0^+(i \sqrt{2\mu + \lambda^2}|x-x_1|)  H_0^+(i \sqrt{2\mu+ \lambda ^2 }|y-y_1|)] \Big| \les k(x,x_1) k(y,y_1), \\
\Big|\partial \big(\frac{\frac{\partial}{\partial\lambda} [ \chi(\lambda) H_0^+(i \sqrt{2\mu + \lambda^2}|x-x_1|)  H_0^+(i \sqrt{2\mu+ \lambda ^2 }|y-y_1|)] }{\lambda}\big)\Big| \les \lambda^{-2}k(x,x_1) k(y,y_1). \label{secondder}
             \end{align}
Hence, using Lemma~\ref{lem:ibp}  with the bounds \eqref{firstder} and \eqref{secondder} we obtain the contribution of $A_4(\lambda,p,q)$ to the $\lambda$-integral in \eqref{r1int} as 
 \begin{align} \label{op2}
      \frac{\|a^2+b^2\|_1 } {64 t } H_0^+(i \sqrt{2\mu}\ p )M_{22} M_{22} H_0^+(i \sqrt{2\mu}\ q) + O \Big( \frac{k(x,x_1) k(y,y_1)} {t^{3/2}}\Big).
          \end{align}                  
 For  $A_2(\lambda,p,q)$, we have
\begin{align}\label{52}
       \begin{split}
         &[h^{+}(\lambda)R_0^{+}(\lambda^2)(x,x_1)-h^{-}(\lambda)R_0^{-}(\lambda^2)(x,x_1)]R_2(\lambda^2)(y_1,y) \\
         &\hspace{5mm}= C J_0(\lambda p) (\log(\lambda)+1) R_2(\lambda^2)(y_1,y) + i \frac{\|a^2+b^2\|_1} {8}Y_0(\lambda p) R_2(\lambda^2)(y_1,y) 
         \end{split}
          \end{align} 
   for some  $ C \in \mathbb{C}$.  
    
Note that we can apply \eqref{m0} to the left side of this sum and replace $G(\lambda,x,x_1)$ with $J_0(\lambda p)$.                                                                                   
Hence, Lemma~\ref{lem:ibp} together with the bounds in \eqref{r2 bound} and Lemma~ \ref{FG} gives us the contribution of the left side to $\lambda$-integral in \eqref{r1int} as $t^{-{\f 54}} \la x \ra \la x_1 \ra k(y,y_1) $.

To find the contribution of  the right side of the sum in (\ref{52}) recall that
$ Y_0(\lambda|x-x_1|)= \chi(\lambda p) [{\f 2 \pi} \log(\frac{\lambda p}{2})+c+\widetilde{O}((\lambda p)^{2}\log(\lambda p))] $. Multiplying this with $R_2(\lambda^2)(y_1,y)$ we have
\begin{multline*}
 \f2{\pi}\log|x-x_1| \chi(\lambda p) R_2(\lambda^2)(y_1,y)+ [\log \lambda +c] \chi(\lambda p) R_2(\lambda^2)(y_1,y) \\ + \widetilde{O}(\lambda p)^{2}\log(\lambda p) \chi(\lambda p)R_2(\lambda^2)(y_1,y).
    \end{multline*}
Using Lemma \ref{cutoff} and (\ref{m0}),  the contribution of the second term to $\lambda$ integral in \eqref{r1int} can be obtained as $ \la x_1 \ra k(y,y_1) t^{-3/2} $ in a similar way as in $A(\lambda,y,y_1) $ in Lemma \ref{lem;S_12}. And the contribution of the third term follows as $ \frac{\la x \ra \la x_1 \ra k(y,y_1)}{t^{3/2}}$ with Lemma~\ref{lem:ibp}. 

Finally, for the first term we take $\mathcal{E}(\lambda,p,q) =  \f2{\pi}\log|x-x_1| \chi(\lambda p) R_2(\lambda^2)(y_1,y)$ and see $\mathcal{E}_1(0,p,q) =  -\f{i} {2 \pi}\log|x-x_1| \chi(\lambda p) H_0^+(i\sqrt 2\mu q)$. Using Lemma~\ref{lem:ibp} with the bounds of $R_2(\lambda)$ the contribution of $A_2(\lambda,p,q)$ is obtained as
\begin{multline} \label{op3}
             \frac{i \|a^2+b^2\|_1 } {32 \pi t} \int_{\R^4}G_0(x,x_1)M_{11} [ v_1 S_1 v_2](x_1,y_1)M_{22} H_0^+(i \sqrt{2\mu}\ q)dx_1dy_1 \\
     + O \Big(  \frac{\la x \ra \la x_1 \ra k(y,y_1)}{t^{1+\alpha}}\Big). 
           \end{multline}
With a similar argument the contribution of $A_3(\lambda,p,q)$ is
 \begin{multline} \label{op4}
     \frac{i\|a^2+b^2\|_1} {32 \pi t} \int_{\R^4} H_0^+(i \sqrt{2\mu}\ p) M_{22} [ v_1 S_1 v_2](x_1,y_1) M_{11}G_0(y,y_1) dx_1 dy_1  \\
     + O \Big(  \frac{\la x \ra \la x_1 \ra k(y,y_1)}{t^{1+\alpha}}\Big) . 
           \end{multline}  
 Adding up \eqref{op1}, \eqref{op2}, \eqref{op3}, \eqref{op4} gives the statement.                    
           \end{proof}
To find the contribution of the terms $SS_1$ and $S_1S$ to \eqref{pacrep} we define
 \begin{align*}
  \mathfrak{R}^+_2(\lambda,p,q):= \mathfrak{R}^+_0(\lambda^2)(x,x_1) \mathfrak{R}^+_0(\lambda^2)(y_1,y)-\mathfrak{R}^{-}_0(\lambda^2)(x,x_1) \mathfrak{R}^{-}_0(\lambda^2)(y_1,y).
      \end{align*}     
\begin{prop}  If $|a(x)|+|b(x)| \les \la x \ra ^{-3/2 -}$, then we have
\begin{align}
\label{msdss}
        \Big|   \int_{\R^4} \int_0^\infty e^{it\lambda^2} \lambda \chi(\lambda) \mathfrak{R}^+_2(\lambda,p,q) [v_1 S_1S v_2](x_1,y_1) d\lambda dx_1 dy_1 - {\f 1 t} \mathfrak{F}_2 (x,y) \Big| \les \frac{ \la x \ra \la y \ra }{t^{1+\alpha}}, 
       \end{align}
\begin{align}
          \Big|   \int_{\R^4} \int_0^\infty e^{it\lambda^2} \lambda \chi(\lambda)  \mathfrak{R}^+_2(\lambda,p,q)v_1SS_1v_2](x_1,y_1) d\lambda dx_1 dy_1 - {\f 1 t} \mathfrak{F}_3 (x,y) \Big| \les \frac{ \la x \ra \la y \ra}{t^{1+\alpha}},
                     \end{align}
where
\begin{align*}
        \mathfrak{F}_2 (x,y)= -\dfrac{1}{4}\int_{\R^4}  \mathcal{G}_0(x,x_1)  v_1(x_1)[S_1S](x_1,y_1)v_2(y_1)M_{11} dx_1 dy_1,   
           \end{align*}
\begin{align*}
        \mathfrak{F}_3 (x,y)= -\dfrac{1}{4}\int_{\R^4}M_{11} v_1(x_1)[SS_1](x_1,y_1)v_2(y_1) \mathcal{G}_0(y,y_1)  dx_1 dy_1.  
    \end{align*}

           \end{prop}
\begin{proof} We consider only (\ref{msdss}). Note that                         
\begin{multline*}  
        \mathfrak{R}^+_2(\lambda,p,q)  
              =[R_0^{+}(\lambda^2)(x,x_1)M_{11}M_{11} R_0^{+}(\lambda^2)(y_1,y)-R_0^{-}(\lambda^2)(x,x_1)M_{11}M_{11} R_0^{-}(\lambda^2)(y_1,y)] \\
             +  [R_0^{+}(\lambda^2)(x,x_1)-R_0^{-}(\lambda^2)(x,x_1)]M_{11}M_{22}R_2(\lambda^2)(y_1,y) \hspace{10mm} \\
             +R_2(\lambda^2)(x,x_1) M_{22}M_{11}[R_0^{+}(\lambda^2)(y_1,y))-R_0^{-}(\lambda^2)(y_1,y))]\\
             = B_1(\lambda,p,q)+B_2(\lambda,p,q)+B_3(\lambda,p,q)  \hspace{40mm}
                          \end{multline*}  
Again a similar kernel to $B_1(\lambda,p,q) $ is examined in the scalar case. It has the contribution 
\begin{align}\label{b1}
    \frac{1} {8 \pi } \int_{\R^4} G_0(x ,x_1) M_{11} [v_1S_1S v_2](x_1,y_1)M_{11}  dx_1dy_1
      + O \Big( \frac{\la x \ra \la x_1 \ra}{t^{1+\alpha}} \Big)
           \end{align}
to the integral in \eqref{msdss}.           
For $B_2(\lambda,p,q) = J_0(\lambda p)M_{11}M_{22}R_2(\lambda^2)(y_1,y) $ we can use the property (\ref{m0}) on the left side of $S_1S$ and exchange $J_0(\lambda p)$ with $G(\lambda , x,x_1)$. Then Lemma~\ref{lem:ibp2} together with the bounds in Lemma \ref{FG} and (\ref{r2 bound}) gives us 
\begin{align}\label{b2}
     \Big|  \int_0^\infty e^{it\lambda^2}\lambda \chi(\lambda) B_{2}(\lambda,p,q)  [v_1S_1S v_2](x_1,y_1) d\lambda \Big| 
      \les  \frac{ \la x_1 \ra k(y,y_1)}{t^{1+\alpha}}.
           \end{align}
Lastly we consider $B_3(\lambda,p,q)= {\f i 2}R_2(\lambda^2)(x,x_1) J_0(\lambda|y-y_1|) \chi(\lambda q)$. Applying  Lemma \ref{lem:ibp}, we have
\begin{align} \label{b3}
        \int_0^\infty e^{it\lambda^2} \lambda \chi(\lambda) B_3(\lambda,p,q) d\lambda 
        = - {\f i {16  t}}  H_0^+(i \sqrt{2\mu}|x-x_1|) M_{22}M_{11}   + O \Big( \frac{k(x,x_1) \la y \ra \la y_1 \ra}{t^{\f 3 2}} \Big) 
        \end{align}
since $ \partial_\lambda J_0(\lambda|y-y_1|) \les \la y\ra \la y_1\ra $  and $\partial^2_\lambda J_0(\lambda|y-y_1|) \les \lambda^{-1} \la y\ra \la y_1\ra $ for $\lambda q \les 1$; and the support of $\chi^\prime(\lambda q) $ allows us to reduce the spatial weight. 
Hence, (\ref{b1}), (\ref{b2}), and (\ref{b3}) establishes the proof. 
                  \end{proof}                                  
The following Proposition will take care of the contributions of the following two integrals:
 \begin{align} \label{mssdss}
   \int_{\R^4} \int_0^{\infty} e^{it\lambda^2} \lambda \chi(\lambda) \mathfrak{R}_3(\lambda,p,q)v_1 SS_1S v_2 (x_1,y_1) d\lambda dx_1 dy_1,
       \end{align}
\begin{align} \label{ms}
      \int_{\R^4} \int_0^{\infty} e^{it\lambda^2} \lambda \chi(\lambda) \mathfrak{R}_3(\lambda,p,q)v_1 S v_2 (x_1,y_1) d\lambda dx_1 dy_1,
       \end{align}
   where 
 $$ \mathfrak{R}_3(\lambda,p,q):=  \frac{\mathfrak{R}_0^+(\mu+\lambda^2)\mathfrak{R}_0^+(\mu+\lambda^2)}{ h_+(\lambda)} -  \frac {\mathfrak{R}_0^{-}(\mu+\lambda^2) \mathfrak{R}_0^{-}(\mu+\lambda^2)}{ h_{-}(\lambda)}. $$            

\begin{prop} [Proposition 5.5 in \cite{EGm}] \label{prop:s} Let $0<\alpha< 1/4$. If $|a(x)|+|b(x)| \les \la x \ra ^{-3/2-\alpha-}$ then for any absolutely bounded operator $\Gamma$ we have
\begin{multline*} 
               \int_{\R^4} \int_0^{\infty} e^{it\lambda^2} \mathfrak{R}_3(\lambda,p,q)v_1 \Gamma v_2 (x_1,y_1) d\lambda dx_1 dy_1 \\          
     = -\frac {1}{4 \|a^2+b^2\|_1} \int_{\R^4} M_{11} v_1 \Gamma v_2 dx_1 dy_1 
        + O\Big( \f{\sqrt{w(x)w(y)}  }{t\log^2(t)} \Big)+O\Big( \f{\la x\ra^{\f12+\alpha+} \la y\ra^{\f12+\alpha+}}{t^{1+\alpha} } \Big). 
         \end{multline*}                              
                    \end{prop}    
 \begin{corollary} Under the same conditions of Proposition~\ref{prop:s} we have
 \begin{align*}
   & |(\ref{mssdss})-{\f 1 t}\mathfrak{F}_4(x,y)| \les  O\Big( \f{\sqrt{w(x)w(y)}  }{t\log^2(t)} \Big)+O\Big( \f{\la x\ra^{\f12+\alpha+} \la y\ra^{\f12+\alpha+}}{t^{1+\alpha} } \Big), \\
    & |(\ref{ms})-{\f 1 t}\mathfrak{F}_5(x,y)| \les  O\Big( \f{\sqrt{w(x)w(y)}  }{t\log^2(t)} \Big)+O\Big( \f{\la x\ra^{\f12+\alpha+} \la y\ra^{\f12+\alpha+}}{t^{1+\alpha} } \Big),  
         \end{align*}
where 
\begin{align*}
    & \mathfrak{F}_4(x,y) = -\frac {1}{4 \|a^2+b^2\|_1} \int_{\R^4} M_{11} v_1(x_1) [SS_1S](x_1,y_1) v_2(y_1)M_{11} dx_1 dy_1, \\
      & \mathfrak{F}_5(x,y) =  -\frac {1}{4 \|a^2+b^2\|_1} \int_{\R^4} M_{11} v_1(x_1) S(x_1,y_1) v_2(y_1) M_{11}dx_1 dy_1 .    
          \end{align*}                                             
                                            \end{corollary}
 The contribution of $ E(\lambda)(x,y) $ can be handled as in Proposition 4.9 in \cite{EGw} and we can obtain the following proposition.
 \begin{prop}Let $0<\alpha< 1/4$. If $|a(x)|+|b(x)| \les \la x \ra ^{-3/2-\alpha-}$, then we have
 \begin{align*}
          \int_0^{\infty} e^{it\lambda^2} \lambda \chi(\lambda) \big[ \mathcal{R}_0^+(\mu+\lambda^2)v_1 E v_2 \mathcal{R}_0^+(\mu+\lambda^2) - \mathcal{R}_0^{-}(\mu+\lambda^2)v_1 E v_2 \mathcal{R}_0^{-}(\mu+\lambda^2) \big] (x,y) d\lambda  \\
    = O\big(\frac{\la x \ra ^{{\f12}+\alpha}  \la y \ra ^{{\f12}+\alpha}}{ t^{1+\alpha}}  \big).                                                                        
                        \end{align*} 
                        \end{prop} 
 We found the boundary terms $ \mathfrak{F}_i( x,y) $, $i=1,..,5$  that has $\f1{t}$ decay for every term appearing in the expansion \eqref{matrixRV}. Also we note that the contribution of free resolvent is calculated in \cite{EGm} as
\begin{align} \label{freematrix}
 \int_0^{\infty} e^{it\lambda^2}\lambda \chi(\lambda)[\mathcal{R}_0^{+}(\mu+\lambda^2)-\mathcal{R}_0^{-}(\mu+\lambda^2)] (x,y)d\lambda= -{\f 1 {4t}} M_{11}+O\Big(\frac{\la x \ra ^{\f 32} \la y \ra ^{\f 32}}{t^{\f 54}}\Big).
     \end{align}   

Considering this and the expansion \eqref{matrixRV} we see that the assertion of Theorem~\ref{t22} is satisfied for 
$$ \mathfrak{F}(x,y) = \mathfrak{F}_0(x,y)+\f1{\|a^2+b^2\|_{L_1(\R^2)}c_0^2} \sum_{i=1}^4 \mathfrak{F}_i(x,y)-\mathfrak{F}_5(x,y). $$ 
 
 The following proposition concludes the explicit representation of $ \mathfrak{F}(x,y)$ in Theorem~\ref{t22}.

\begin{prop} Under the conditions of Theorem ~\ref{main2} we have
$$ \mathfrak{F}(x,y) = -\f 1{4 c_0^2}
\psi(x)\sigma_3 \psi(y) $$ where $(\mathcal{H}_0-\mu I)\psi = 0 $ in the sense of distribution and $ \psi \in L^{\infty}(\R^2) \times L^{\infty}(\R^2).$
      \end{prop} 
\begin{proof} By definition of $S_1$, it has the kernel $ S_1= \phi^T(x) \phi (y)$. Using this in the operator obtained as $\mathfrak{F}_1$ we have 
\begin{multline} 
\label{m1}
                         \mathfrak{F}_1( x,y) = - \frac{ \|a^2+b^2\|_1} {4} \int _{\R^2} \mathcal{G}_0(x, x_1) [v_1 \phi](x_1) dx_1 \int _{\R^2}[\phi v_2](y_1) \mathcal{G}_0 (y_1,y)dy_1 
 \\
    =  \frac{ \|a^2+b^2 \|_1  } { 4} \big( (-c_0,0)-\psi_2(x) \big) \big( \psi_1(y) - (c_0,0) \big). 
                      \end{multline} 
For the second equality we used Lemma~ \ref{G0matrix}.

For $\mathfrak{F}_2(x,y)$ recall the equality $Qv_2M_{11} = 0$. With a similar calculation in scalar case we conclude that  $S_1Sv_2M_{11} = -S_1Tv_2M_{11} $. Then using the definition of $S_1$ and Lemma~\ref{G0matrix}, we obtain                   
\begin{multline} 
\label{m2}
                          \mathfrak{F}_2( x,y) =-\dfrac{1}{4}\int_{\R^4}  \mathcal{G}_0(x,x_1)  v_1(x_1)[S_1S](x_1,y_1)v_2(y_1)M_{11} dx_1 dy_1\\
                             =  \f  1 {4} \int _{\R^4} \mathcal{G}_0(x,x_1) v_1(x_1)\phi (x_1) [\phi T](y_1) v_2 (y_1)M_{11} dx_1 dy_1\\= 
 \f  1 {4}  \int _{\R^2}   [\mathcal{G}_0 v_1\phi](x, x_1) dx_1 \ \la (a,b) , T\phi \ra (1,0)\\
  =  \frac{\|a^2+b^2\|_1}{4} (c_0,0)\big( (c_0,0) - \psi_1 (x)  \big). 
                                            \end{multline} 
For the third equality note that
$$  [T\phi] v_2 M_{11} = T\phi  \left[\begin{array}{cc} a & 0 \\ b & 0 \end{array} \right]=[(a,b) \cdot T\phi] (1,0). $$
Consequently, using
        $$ M_{11} v_1 [T\phi] = T\phi  \left[\begin{array}{cc} a & 0 \\ b & 0 \end{array} \right]=[(a,b) \cdot T\phi] (1,0)^T 
        $$
 we have                                           
 \begin{align}
 \label{m3}
                      \mathfrak{F}_3( x,y) =  \frac{ \|a^2+b^2\|_1  } { 4} \big( (c_0,0)+\psi_2(y) \big) (c_0,0).
                                           \end{align}
Again as in the scalar case using the definition of $S$, $S_1$, and the equality \eqref{m0} we have $M_{11}v_1SS_1Sv_2M_{11} = M_{11}v_1TS_1Tv_2M_{11} $, then using the definition of $S_1$ we have   
  \begin{multline}
 \label{m4}
                          \mathfrak{F}_4( x,y) =  {\f 1 {4 \|a^2+b^2\|_1}}  \int _{\R^2} M_{11}v_1(x_1)[T\phi](x_1) [\phi T] (y_1) v_2(y_1)M_{11} dx_1 dy_1\\
                          = \frac{  \|a^2+b^2\|_1} {4}[ \la (a,b) , T\phi \ra (1,0) \cdot \la (a,b) , T\phi \ra (1,0) ]
                           = \frac{  \|a^2+b^2\|_1} {4} c_0^2 .  
                           \end{multline}

Finally we have $M_{11}v_1Sv_2M_{11} = M_{11}v_1 v_2M_{11} $. Using this,  $\mathfrak{F}_4( x,y)$ can be written explicitly as \begin{multline} \label{m5} 
                          \mathfrak{F}_5( x,y) = - {\f 1 {4 \|a^2+b^2\|_1}}  \int _{\R^2} M_{11}v_1 v_2M_{11} dx_1 dy_1
                           = {\f 1 {4 \|a^2+b^2\|_1}} \|a^2+b^2\|_1 = {\f 1 4} M_{11}.
                                      \end{multline}                             

Multiplying \eqref{m1}, \eqref{m2}, \eqref{m3}, \eqref{m4}, \eqref{m5} with required constants and summing up together with the boundary term \eqref{freematrix} from the free resolvent for matrix Schr\"dinger operator we obtain
$$  \sum_{i=0}^5 \mathfrak{F}_i(x,y) =  - \frac{ 1 } {4 c_0^2} \psi_2(x) \psi_1(y)= - \frac{ 1 } {4 c_0^2} \sigma_3 \psi_1(x)  \psi_1(y).  $$ 
\end{proof}                              
                                                                                                     
                                                                                                                                                                        \subsection{Proof of the Theorem \ref{t21}}  \hspace{3mm}\\                                            
The ${\f 1 t }$  bound for the free resolvent, for a similar error term to $E$, and for the term $h_{\pm}(\lambda)^{-1}S$ were examined in \cite{EGm} in Proposition~5.4, Proposition~7.5, and Proposition~7.2 respectively. Since the proof of Proposition~7.2 requires the operator $S$ only to be absolutely bounded it can be extended to the term $h_{\pm}(\lambda)^{-1}SS_1S$.    

For the operators $QD_0Q$, $SS_1$, and $S_1S$ recall the expansion:
\begin{multline*}
 \mathfrak{R}^+_0(\lambda^2)(x,x_1) \mathfrak{R}^+_0(\lambda^2)(y_1,y)-\mathfrak{R}^{-}_0(\lambda^2)(x,x_1) \mathfrak{R}^{-}_0(\lambda^2)(y_1,y)\\
             = B_1(\lambda,p,q)+B_2(\lambda,p,q)+B_3(\lambda,p,q).  
                          \end{multline*}
 The ${\f 1 t }$  bound for a similar kernel to $ B_1(\lambda,p,q) $ is established in Proposition~3.11 in \cite{EG} for the operator $QD_0Q$, $SS_1$, and $S_1S$. Furthermore, Proposition~7.2 in \cite{EGw} shows that $ B_2(\lambda,p,q)$ and $B_3(\lambda,p,q)$ can be also estimated by ${\f 1 t }$ for the operator $QD_0Q$. Since the proof of Proposition~7.2  requires the operator $QD_0Q$ only to be absolutely bounded it can be adopted to $SS_1$ and $S_1S$. 

 Hence, it is enough to establish the ${\f 1 t }$ bound for the operator  $h_{\pm}(\lambda) S_1$. The following Proposition will conclude Theorem \ref{t21} 
\begin{prop} If $|a(x| +|b(x)|\les \la x \ra ^{-1-} $ then we have,                                            
     $$   \int_{\R^4} \int_0^\infty e^{it\lambda^2} \lambda \chi(\lambda) \mathfrak{R}^+_1(\lambda,p,q) [v_1S_1 v_2] (x_1,y_1) d\lambda dx_1 dy_1 = O \Big({\f1t }\Big). $$ 
        \end{prop} 
Recall the calculation;
 $$\mathfrak{R}_1(\lambda,p,q)= A_1(\lambda,p,q)+ A_2(\lambda,p,q)+ A_3(\lambda,p,q)+ A_4(\lambda,p,q). $$
             
Not that Theorem~3.1 in \cite{EG} establishes the $1/t$ bound for a similar operator to $ A_1(\lambda,p,q) $. Using \eqref{m0} one can adopt the same proof to $ A_1(\lambda,p,q) $.

 Using the bounds \eqref{r2 bound}, the contribution of $A_4(\lambda,p,q)= {\f i2} R_2(x,x_1)M_{22}M_{22} R_2(y_1,y) $ can be handled as
\begin{multline*} 
       \int_0^\infty e^{it\lambda^2} \lambda \chi(\lambda) {\f i2} R_2(\lambda^2)(x,x_1)R_2(\lambda^2)(y_1,y) d\lambda \\
       \les {\f 1t} \int_0^{2\lambda_1} \big| \partial_{\lambda}[R_2(\lambda^2)(x,x_1)R_2(\lambda^2)(y_1,y)]\big| d\lambda 
             \les k(x,x_1)k(y,y_1)  O \Big({\f1t }\Big).
      \end{multline*}
The assertion for $A_4(\lambda,p,q)$ follows with $\|v_1(x_1) (k(x,x_1)\|_{L^2_{x_1}} \les 1$.
   
To prove the contribution of the operators $A_2(\lambda,p,q)$ and $A_3(\lambda,p,q) $  we need the following Lemma. 
\begin{lemma} Under the same conditions of the previous proposition we have, 
$$    \int_{\R^4} \int_0^\infty e^{it\lambda^2} \lambda \chi(\lambda) A_2(\lambda,p,q) v_1(x_1) [S_1](x_1,y_1) v_2(y_1) d\lambda dx_1 dy_1 = O \Big({\f1t }\Big). $$ 
The same bound is valid for $A_3(\lambda,p,q) $.
\end{lemma}
\begin{proof}
We have to consider the large and the small energy contribution separately.\\
Case 1: $ \lambda|x-x_1| \les 1 $. Recall that 
$A_2(\lambda,p,q) = C J_0(\lambda p) (\log(\lambda)+1) R_2(\lambda^2)(y_1,y) + z Y_0(\lambda p) R_2(\lambda^2)(y_1,y)$ for some $C \in \R$ and $z \in \mathbb{C} $.
Taking this expansion and the projection property \eqref{m0} of $S_1$ into account  it is enough to consider the contribution of the following two integrals
\begin{align} \label{s1}
   \int_0^\infty e^{it\lambda^2} \lambda \chi(\lambda) F(\lambda,x,x_1) R_2(\lambda^2)(y_1,y) d\lambda, 
   \end{align}
\begin{align} \label{s2}
   \int_0^\infty e^{it\lambda^2} \lambda \chi(\lambda) \log(\lambda)G(\lambda,x,x_1) R_2(\lambda^2)(y_1,y) d\lambda.  
   \end{align} 
By Lemma \ref{FG}  and integration by part once, we have 
\begin{multline*}
\begin{split}
& | \ref{s1}| \les {\f1 t} \int_0^\infty e^{it\lambda^2} \chi^{\prime}(\lambda)  F(\lambda,x,x_1) R_2(\lambda^2)(y_1,y) d\lambda\\
                          &\hspace{10mm}+ {\f1 t} \int_0^{2\lambda_1}\Big|\partial_{\lambda} F(\lambda,x,x_1) R_2(\lambda^2)(y_1,y)\Big| d\lambda 
                        + {\f1 t} \int_0^{2\lambda_1} \Big|F(\lambda,x,x_1) \partial_{ \lambda} R_2(\lambda^2)(y_1,y)\Big| d\lambda\\
                        &\hspace{10mm}\les  {\f{k(y,y_1)} t} \int_0^{2\lambda_1}  |F(\lambda,x,x_1)|+\big|\partial_{\lambda} F(\lambda,x,x_1)\big| d\lambda \les \frac{ k(y,y_1)k(x,x_1)}{t}.
                         \end{split}
                          \end{multline*}
 With a similar argument  $\big| \ref{s2}\big| \les \frac{ \la x_1 \ra ^{1/2}k(y,y_1)}{t}$.
 
Case 2: $ \lambda|x-x_1| \gtrsim 1 $. For this case we give a sketch of the proof and refer Lemma~3.8 in \cite{EG} to the reader for details. 

Note that using \eqref{m0} the $\lambda$-integral of
  \begin{align*}
  \int_{\R^4} \int_0^{\infty} e^{it\lambda^2} \lambda \chi(\lambda) \log(\lambda)\widetilde{J}_0(\lambda p)M_{11} [v_1Sv_2] M_{22} (x_1,y_1) R_2(\lambda^2)(y_1,y) d\lambda dx_1dy_1
         \end{align*}
 can be written as 
\begin{align} \label{int:large}
 \int_0^{\infty} e^{it\lambda^2} \lambda \chi(\lambda) \log(\lambda)[\widetilde{J}_0(\lambda p)- \widetilde{J}_0(\lambda (1+|x|)] R_2(\lambda^2)(y_1,y) d\lambda. 
    \end{align}
 Let $s= \max(|x-x_1|, 1+|x|)$ and $r= \min(|x-x_1|, 1+ |x|)$. Using the large energy representation \eqref{largr} of Bessel functions and pulling the slower oscillation $e^{\pm i \lambda r}$ out, \eqref{int:large} can be rewritten as the sum of  
 \begin{align} \label{last} 
  \int_0^{\infty} e^{it  (\lambda^2 \pm \lambda r t^{-1}) } \lambda \chi(\lambda) \log(\lambda) \widetilde{G}_\pm(\lambda,s,r) R_2(\lambda^2)(y_1,y) d\lambda, 
   \end{align}  
  where  
  $$  \widetilde{G}^{\pm}(\lambda,s,r) := \widetilde{\chi}(\lambda s) \omega_{\pm}(\lambda s) - e^{{\pm}i\lambda (s-r)}\widetilde{\chi}(\lambda r) \omega_{\pm}(\lambda r).  $$  

 We finish the proof by recalling that in Lemma~3.8 of \cite{EG}, a similar integral to \eqref{last} is bounded by ${\f 1t}$. We note that the only difference between integral \eqref{last} and the one examined in \cite{EG} is that $ R_2(\lambda^2)(y_1,y) $ is replaced with $F(\lambda,x,x_1)$.
    
   \end{proof}

\end{document}